\documentclass[11pt]{article}

\usepackage[T1]{fontenc}
\usepackage{amsmath,amssymb,amsthm,mathtools}
\usepackage[a4paper,margin=30mm]{geometry}
\usepackage[colorlinks=true,allcolors=blue,hypertexnames=false]{hyperref}

\newtheorem{theorem}{Theorem}[section]
\newtheorem{lemma}[theorem]{Lemma}
\newtheorem{proposition}[theorem]{Proposition}
\newtheorem{corollary}[theorem]{Corollary}
\theoremstyle{definition}
\newtheorem{definition}[theorem]{Definition}
\newtheorem{example}[theorem]{Example}

\newtheorem*{openproblem}{Open problem}
\theoremstyle{remark}
\newtheorem{remark}[theorem]{Remark}

\newcommand{\T}{\mathbb T}
\newcommand{\C}{\mathbb C}
\newcommand{\R}{\mathbb R}
\newcommand{\Z}{\mathbb Z}
\newcommand{\N}{\mathbb N}
\newcommand{\Q}{\mathbb Q}
\newcommand{\CP}{\mathbb{CP}^1}
\newcommand{\eu}{\mathrm{e}}
\newcommand{\ii}{\mathrm{i}}
\newcommand{\co}{\mathfrak{m}}
\DeclareMathOperator{\Pic}{Pic}
\newcommand{\Proj}[1]{\mathbb{P}(#1)}

\DeclareMathOperator{\adj}{adj}
\DeclareMathOperator{\diag}{diag}
\newcommand{\ip}[2]{\langle #1,\, #2\rangle}
\newcommand{\norm}[1]{\lVert #1\rVert}
\newcommand{\abs}[1]{\lvert #1\rvert}

\title{A Chern-Class Obstruction to Measurable Eigensections\\
of Analytic Quasi-Periodic Bundle Cocycles}
\author{Ahmadreza Azimifard}
\date{August 2026}
\hypersetup{pdfauthor={Ahmadreza Azimifard}}

\makeatletter
\renewcommand{\@maketitle}{%
  \newpage
  \null
  \vskip 2em%
  \begin{center}%
    {\LARGE \@title \par}%
    \vskip 1.5em%
    {\large \@author \par}%
    \vskip 1em%
    {\large \@date}%
  \end{center}%
  \par
  \vskip 1.5em}
\makeatother

\begin{document}
\maketitle

\begin{abstract}
We consider linear cocycles over an ergodic translation of the two-torus, acting on a
Hermitian vector bundle presented by unitary sewing matrices, together with a continuous
invariant complex line subbundle $L$. Our main result is line-resolved: if $c_1(L)\neq 0$,
then, under a measurable phase-coboundary hypothesis---which we prove to be automatic
for rank-two analytic cocycles with a dominated splitting over a Diophantine
translation---$L$ carries no nonzero measurable eigensection, for any eigenvalue. The
abstract obstruction holds in every rank $n\ge1$; the analytic discharge of the
hypothesis is a rank-two theorem. In the analytic dominated Diophantine regime the
obstruction is moreover \emph{exact}: writing $k$ for the winding vector of the
multiplier's periodic phase, $L$ carries a nonzero measurable eigensection if and only if
$c_1(L)=0$ and $k=0$, in which case the eigenvalues realised on $L$ form a dense coset
$\lambda_0\{\eu^{-2\pi\ii\nu\cdot\tau}\}$ of the circle of radius $\eu^{\lambda_L}$, all
geometrically simple within $L$, i.e.\ with one-dimensional corresponding eigensection
spaces carried by $L$, and with real-analytic nowhere-vanishing eigensections. The
Diophantine hypothesis is
not removable: over a Liouville translation there are analytic dominated data with
$c_1(L)=k=0$ carrying nothing, the obstruction lying in the modulus of the multiplier
rather than in its phase. The mechanism is a normalization on the
universal cover which converts the first Chern class into the magnetic charge of a
scalar quasi-periodicity class; the eigensection equation becomes a fixed-point equation
for a unitary operator lying in an irrational magnetic-translation family, and the
classical covariance argument (one eigenvector generates uncountably many mutually
orthogonal ones) yields the contradiction. The covariance mechanism itself is classical;
what appears to be new is the identification of $c_1$ of an invariant subbundle---an
$H^2$-invariant, invisible over one-frequency bases---as a magnetic charge obstructing
measurable eigensections, together with the Stein-theoretic construction of an analytic
normalized gauge on nontrivial line bundles over strip tori that discharges the
phase-coboundary hypothesis. The result does not by itself exclude point spectrum of the
cocycle: eigensections may exist, and do exist in examples, on topologically trivial
invariant lines; the charge and winding vector of a line decide whether it contributes to the point
spectrum at all, while the location of its contribution is fixed by the mean multiplier
data through $\lambda_0$. Equal-exponent (elliptic), nonuniformly hyperbolic, and
noninvertible regimes are not treated here.
Selected statements have been formalized and kernel-checked in a private Lean development,
including the main obstruction at the integer-charge level. The conventional
first-Chern-class-to-charge identification and the Stein/Cousin-II input remain external,
and no claim of complete formal verification is made.
\end{abstract}

\section{Introduction}\label{sec:intro}

Let $\tau=(\alpha,\beta)\in\R^2$ and let $S\colon w\mapsto w+\tau$ be the corresponding
translation of the two-torus $\T^2=\R^2/\Z^2$, assumed ergodic (equivalently:
$1,\alpha,\beta$ are rationally independent). A \emph{quasi-periodic bundle cocycle} is a
continuous invertible-matrix-valued map $A$ on $\R^2$, equivariant with respect to a
unitary sewing presentation of a Hermitian vector bundle $V\to\T^2$, so that the pair
$(S,A)$ acts on measurable sections of $V$ by the weighted translation
$(T_AF)(w)=A(w-\tau)F(w-\tau)$. Such operators arise as transfer operators of
quasi-periodic Schr\"odinger-type and Jacobi-type models, as matrix weighted shifts, and
---the application that motivated this work---as vector-valued Zak transforms of
finite Weyl polynomials in time-frequency analysis (Section~\ref{sec:zak}).

The eigenvalue equation for $T_A$ is the \emph{eigensection equation}
\begin{equation}\label{eq:eigen}
A(w)\,F(w)=\lambda\,F(w+\tau)\qquad\text{for a.e.\ }w\in\R^2 ,
\end{equation}
for a measurable section $F$ of $V$ and $\lambda\in\C$. A nonzero measurable solution
spans, almost everywhere, a measurable invariant line; when the cocycle admits a
dominated splitting $E^s\oplus E^u$ by continuous invariant line subbundles, every
nonzero measurable eigensection is in fact carried by $E^s$ or by $E^u$ almost everywhere
(Lemma~\ref{lem:carried}). The question addressed here is: \emph{which invariant lines
can carry an eigensection?}

Over a one-frequency base $\T^1$ the topological invariants available for this question
are of $H^1$-type: the winding (degree) of the multiplier of an invariant line, and the
degree-type invariants of the cocycle map itself. These control ergodicity and spectral
properties of skew products and weighted shifts in a classical body of work (Anzai
\cite{Anzai1951}, Furstenberg \cite{Furstenberg1961}, Gabriel--Leman\-czyk--Liardet
\cite{GLL1991}, Iwanik--Leman\-czyk--Rudolph \cite{ILR1993}, Fr\k{a}czek
\cite{Fraczek2000,Fraczek2004}), and they enter the modern analytic theory of
one-frequency cocycles through the winding of the invariant multiplier
(Avila--Jitomirskaya--Sadel \cite{AJS2014}). Over a two-dimensional base a genuinely new
invariant appears: an invariant line subbundle $L\subset V$ has a first Chern class
$c_1(L)\in H^2(\T^2;\Z)\cong\Z$, which has no one-frequency counterpart. To our
knowledge, no obstruction to point spectrum based on this $H^2$-invariant appears in
the literature for bundle cocycles, beyond the classical homogeneous (nilflow) case
discussed in the related-work section below; the closest topological results for
quasi-periodic cocycles over higher tori concern the existence of dominated
splittings (Duarte--Klein \cite{DuarteKlein2019}) or label spectral gaps by
$H^1$-type Schwartzman data (Li--Wu \cite{LiWu2025}), and the conceptually adjacent
Wannier-bundle literature obstructs \emph{well-localized frames} of spectral
subspaces rather than eigensections (see below). In the opposite, constructive
direction, the very recent four-point HRT construction \cite{DDSWY2026} builds a
smooth eigensection on a dominated invariant line certified to be topologically
\emph{trivial}---the case our obstruction leaves open---and is compared in the
related-work section.

\subsection*{Main results}

Fix the conventions of Section~\ref{sec:setting}. Our first main theorem is abstract and
conditional on a single cohomological hypothesis. For a continuous invariant line $\ell$
with normalized unit section $e$ of charge $m$ (Section~\ref{sec:charge}) let
$q(w)=\ip{e(w+\tau)}{A(w)e(w)}$ be its multiplier, $p=(q/\abs q)\,\eu^{2\pi\ii m\alpha
w_2}$ its $\Z^2$-periodic phase, $k\in\Z^2$ the winding vector of $p$ and $g\colon
\T^2\to\R$ its continuous zero-winding phase, so that $p=\eu^{2\pi\ii(k\cdot w+g(w))}$.

\begin{quote}
\textbf{Theorem A} (Theorem~\ref{thm:A}). \emph{Let $n\ge1$, let $(U_1,U_2)$ be a
continuous unitary sewing system on $\R^2$ with values in $U(n)$,
$A\colon\R^2\to GL(n,\C)$ a continuous equivariant cocycle, and assume
$1,\alpha,\beta$ are rationally independent over $\Q$. Let $\ell$ be a continuous
invariant line with $c_1(L)\neq0$, and assume the measurable phase-coboundary hypothesis
\textup{(H2)}: $g-\langle g\rangle$ is a measurable additive coboundary of the rotation
$S$. Then for every $\lambda\in\C$, every measurable section $F$ of $V$ satisfying
\eqref{eq:eigen} and $F(w)\in\ell(w)$ a.e.\ vanishes almost everywhere.}
\end{quote}

Theorem~A requires neither domination, a Lyapunov gap, a Diophantine condition, nor an
$L^2$ hypothesis, and it applies to every rank $n\geq1$ and every $\lambda\in\C$; the
case $\lambda=0$ is immediate from the invertibility of $A$, and ergodicity enters only
in the indicated step. Theorem~B verifies (H2) for analytic
dominated rank-two cocycles using the M\"obius-contraction and Stein-normalization
arguments below. No higher-rank analogue is asserted.

\begin{quote}
\textbf{Theorem B} (Theorem~\ref{thm:B}). \emph{Let $(U_1,U_2,A)$ be strip-admissible
analytic data \textup{(Definition~\ref{def:stripadm})} with values in $U(2)$ and
$GL(2,\C)$, admitting a dominated splitting $E^s\oplus E^u$ by continuous invariant line
subbundles \textup{(Definition~\ref{def:dom})}, over a translation $\tau$ satisfying the
Diophantine condition \textup{DC} \textup{(Definition~\ref{def:DC})}. If $L\in\{E^s,E^u\}$
has $c_1(L)\neq0$, then $L$ carries no nonzero measurable eigensection, for any
$\lambda\in\C$.}
\end{quote}

\noindent
The engine of Theorem~B is a lemma with \emph{no} hypothesis on the Chern class
(Lemma~\ref{lem:analytic}): either line of the splitting admits a real-analytic
normalized unit section, in which $\log\abs q$ and $g$ are real analytic and both
cohomological equations have real-analytic solutions. Because it is charge-free, the same
lemma drives the converse.

\begin{quote}
\textbf{Theorem D} (Theorem~\ref{thm:dich}). \emph{Assume the strip-admissible rank-two data,
dominated splitting, and Diophantine translation of Theorem~B. Let $L$ be any continuous
invariant line---necessarily $E^s$ or $E^u$
\textup{(Lemma~\ref{lem:onlytwo})}. Then $L$ carries a nonzero measurable eigensection if
and only if $c_1(L)=0$ and $k=0$. When both vanish, the eigenvalues realised on $L$ are
exactly $\lambda_0\{\eu^{-2\pi\ii\nu\cdot\tau}:\nu\in\Z^2\}$ with
$\lambda_0=\exp(\langle\log\abs q\rangle+2\pi\ii\langle g\rangle)$, a dense subset of the
circle of radius $\eu^{\lambda_L}$; each is geometrically simple within $L$, i.e.\ its
corresponding eigensection space carried by $L$ is one-dimensional, and every measurable
eigensection carried by $L$ is a.e.\ real analytic and nowhere vanishing
\textup{(Theorem~\ref{thm:class})}.}
\end{quote}

\noindent
The Diophantine hypothesis cannot be deleted from the sufficiency direction of
Theorem~D: over a Liouville translation there are analytic dominated data with
$c_1(L)=0$ and $k=0$ whose line carries no measurable eigensection
(Theorem~\ref{thm:sharp}). We prove nothing about the necessity of the remaining
hypotheses --- domination, analyticity, invertibility, rank two --- and
Example~\ref{ex:elliptic} shows the conclusion can survive the failure of one of them. What fails in Theorem~\ref{thm:sharp} is not the phase
hypothesis \textup{(H2)}, which holds trivially, but the cohomological equation for
$\log\abs q$---the half of the multiplier that the obstruction direction discards
(Remark~\ref{rem:asym}). The two directions are therefore not mirror images at the
abstract level; they become one only in the analytic--Diophantine regime.

Combining Theorem~B with a Birkhoff-type classification of carried lines gives an
operator-level statement:

\begin{quote}
\textbf{Corollary C} (Corollary~\ref{cor:C}). \emph{Assume the strip-admissible rank-two
data, dominated splitting, and Diophantine translation of Theorem~B. Then every nonzero
measurable eigensection of $T_A$ is carried a.e.\ by $E^s$ or $E^u$.
Consequently, if both $c_1(E^s)\neq0$ and $c_1(E^u)\neq0$, then $T_A$ has no eigenvalues
at all; if exactly one vanishes, every eigensection is carried a.e.\ by the topologically
trivial line. With Theorem~D the two invariants of each line decide which lines
contribute to $\sigma_{\mathrm{pt}}(T_A)$; the coset each contributes is located by
$\lambda_0$, which these invariants do not determine
\textup{(Corollary~\ref{cor:D})}.}
\end{quote}

The conclusion is genuinely line-resolved and cannot be upgraded to a bundle-level
statement: on the same nontrivial bundle ($c_1(V)[\T^2]=1$), an analytic \emph{unitary}
equivariant cocycle with equal Lyapunov exponents can have a dense family of analytic
eigensections, all carried by a trivial invariant line, and so can an analytic dominated
cocycle (Examples~\ref{ex:elliptic}--\ref{ex:dominated}). Example~\ref{ex:both} shows
Corollary~C is not vacuous.

\subsection*{What is classical and what appears to be new}

The final contradiction in Theorem~A is a covariance argument of classical lineage: a
unitary $W$ satisfying $WM_s=\chi(s)M_sW$ for a one-parameter unitary family $M_s$ and a
nonconstant character family $\chi$ cannot have an eigenvector in a separable Hilbert
space, since one eigenvector would generate uncountably many eigenvectors with distinct
eigenvalues. This is the Weyl-commutation mechanism (Weyl, von Neumann
\cite{vonNeumann1931}); its incarnations include Zak's magnetic translation groups
\cite{Zak1964} and the translation-covariance proof of absence of bound states for Stark
Hamiltonians (Avron--Herbst \cite{AvronHerbst1977}); the one-frequency, charge-zero limit
of our phase equation is the classical coboundary obstruction of Anzai skew products
\cite{Anzai1951,Furstenberg1961}. This covariance mechanism is classical.

What appears to be new is: (i) the observation that after a universal-cover
normalization, the first Chern class of the \emph{carrying line}---not of the ambient
bundle---becomes precisely the magnetic charge $m$ of the scalar quasi-periodicity class
of the eigensection equation, so that $c_1(L)\neq0$ places the phase walk in an
irrational magnetic translation family regardless of eigenvalue and of multiplier
winding (Sections~\ref{sec:charge}--\ref{sec:magnetic}); (ii) the resulting theorem
itself, an $H^2$-obstruction to measurable eigensections, which requires at least two
base frequencies and whose only antecedents known to us are the classical
homogeneous nilflow case and a nonrigorous physics counterpart
(Section~\ref{sec:related}); and
(iii) the exactness of the obstruction in the analytic dominated Diophantine regime
(Theorem~\ref{thm:dich}): the two invariants $m$ and $k$ are, respectively, the
obstruction to trivializing $L$ and the obstruction to a global logarithm of the
multiplier once $L$ is trivialized (Lemma~\ref{lem:log}), and once both vanish only a
small-divisor problem remains, which the Diophantine condition solves---so the list of
obstructions is complete; and (iv) the analytic normalization (Section~\ref{sec:stein}): a dominated invariant line of
a strip-admissible analytic cocycle extends holomorphically to a strip, and Stein theory
for the quotient---a product of annuli, on which $\Pic\cong H^2$ by Cartan's Theorem~B
and the exponential sequence---produces a \emph{real-analytic} unit section with the
standard charge automorphy for \emph{any} value of the Chern class; this is what makes
the phase data analytic and the small-divisor discharge of (H2) possible.

We emphasize the limits of the conclusion. Theorem~D is a statement about the
strip-admissible dominated regime over a Diophantine translation, and Theorem~\ref{thm:sharp}
shows the arithmetic hypothesis is not decoration. Nothing is proved in the elliptic,
nonuniformly hyperbolic, or noninvertible regimes, nor in rank $n>2$, and nothing here
decides linear independence of any time-frequency configuration; these questions lie
outside the scope of the present paper.

\subsection*{Related work}\label{sec:related}

\emph{Topological invariants of quasi-periodic cocycles.} Duarte--Klein
\cite{DuarteKlein2019} obstruct the existence of dominated splittings in prescribed
homotopy classes of cocycles over higher-dimensional tori; their invariant is the
homotopy class of the cocycle map and their conclusion concerns continuous invariant
subbundles, not eigensections. Avila--Jitomirskaya--Sadel \cite{AJS2014} develop the
one-frequency theory; their Lemma~6.4 and its proof identify the acceleration with
minus the winding of the holomorphic invariant multiplier, an $H^1$-statement that our
winding vector $k$ generalizes but that cannot see $c_1$. Li--Wu \cite{LiWu2025} label
spectral gaps of
symplectic cocycles over ergodic bases by the Schwartzman group, again $H^1$-data.
Herman-type analytic examples and Bjerkl\"ov's $C^2$ two-frequency Schr\"odinger examples
show that positive Lyapunov exponent need not by itself yield uniform hyperbolicity
\cite{Herman1983,Bjerklov2007}.

\emph{Skew products and weighted shifts over rotations.} For circle extensions
$T_\varphi(x,y)=(x+\alpha,y+\varphi(x))$ with $\deg\varphi\neq0$, ergodicity for every
irrational $\alpha$ holds for absolutely continuous $\varphi$ \cite{GLL1991}, with
countable Lebesgue spectrum in the nonzero fibers when $\varphi'$ has bounded variation
\cite{ILR1993}; consequently a $C^1$ circle-valued cocycle of nonzero degree is never
measurably cohomologous to a constant. These are one-frequency, $H^1$-type antecedents
of Proposition~\ref{prop:winding}. For $SU(2)$-valued cocycles, nonzero degree forces
non-ergodicity of the skew product and measurable reduction to the maximal torus
(Fr\k{a}czek \cite{Fraczek2000,Fraczek2004})---a useful contrast showing that
higher-rank degree phenomena need not exclude reducibility. The dictionary between
eigenvalues of matrix weighted shifts and invariant subbundles with multipliers is
classical \cite{Antonevich1996,LatushkinStepin1991}.

\emph{The Wannier/Bloch-bundle line.} The statement ``$c_1\neq0$ forbids well-localized
sections'' has a well-developed history for spectral projections of periodic and
magnetic Schr\"odinger operators: Thouless \cite{Thouless1984}; triviality under
time-reversal (Panati \cite{Panati2007}); the localization--topology correspondence of
Monaco--Panati--Pisante--Teufel \cite{MPPT2018} ($\langle X^2\rangle<\infty$ for a
composite Wannier basis iff the Bloch bundle is trivial); Parseval-frame substitutes in
the nontrivial case (Cornean--Monaco--Moscolari \cite{CMM2019}); non-periodic extensions
\cite{MMP2023,LuStubbs2024}; and, on the covariant-operator side, the
localization dichotomies of Bellissard--van~Elst--Schulz-Baldes \cite{BES1994} and the
delocalization theorems for random Landau Hamiltonians \cite{GKS2007}. These concern
frames of infinite-dimensional spectral subspaces and quantitative decay in physical
space; none excludes \emph{measurable} eigensections of a cocycle over an ergodic torus
translation. In time-frequency analysis, the Balian--Low theorem is a Chern-number
$\pm1$
obstruction at critical density whose classical proof runs through the forced zero (or
winding) of a continuous Zak transform \cite{BHW1995}; our Theorem~A may be regarded as a
measurable-category cousin in which continuity is replaced by ergodicity of the base
dynamics.

\emph{The homogeneous antecedent: nilflows.} The oldest $H^2$-flavored instance of
the mechanism of Theorem~A is classical: $L^2$ of a Heisenberg nilmanifold
decomposes under the central circle action into character subspaces, and the $m$-th
character subspace is canonically a space of sections of a line bundle of degree $m$
over the base $\T^2$ (the Weil--Brezin--Zak picture). Green's theorem on the spectra
of ergodic nilflows \cite{AGH1963}---see also Parry \cite{Parry1969} for the
affine/discrete-time case, and Anzai \cite{Anzai1951} for the skew-product germ---%
gives countable Lebesgue spectrum, in particular \emph{no} eigenfunctions, in every
nonzero central character. In our language this is the conclusion of Theorem~A for a
special homogeneous cocycle in the model charge gauge (exact model multipliers,
$k=0$, constant $A$ along the model sewing). Theorem~A is the inhomogeneous
bundle-cocycle version: arbitrary continuous unitary sewing, arbitrary invertible
equivariant $A$ of any rank, arbitrary multiplier winding, arbitrary eigenvalue,
and merely measurable eigensections---at the price of the coboundary hypothesis
(H2), which Theorem~B discharges in the analytic dominated regime.

\emph{The physics antecedent: quasiperiodically driven qubits.} In the physics
literature, Crowley--Martin--Chandran \cite{CMC2019} (following Martin--Refael--%
Halperin \cite{MRH2017}) study two-tone driven qubits---in our language, smooth
$SU(2)$-valued two-frequency cocycles over a Kronecker flow---and classify them by
the Chern number $C_j$ of a quasi-energy band over the torus of drive phases. Three
of their findings frame ours. First, in the commensurate-approximant limit and
under an explicitly stated differentiability hypothesis on the quasi-energy, they
derive the dispersion identity $\nabla_{\theta_0}\varepsilon_j=
(C_j/2\pi)(-\Omega_2,\Omega_1)$ (their Eq.~(27)), which already precludes a
band-constant quasi-energy when $C_j\neq0$. Second, they show---constructively, in
the smooth category, with an irrationality-measure (Diophantine-type) ingredient
entering exactly where our small divisors do---that a band admits a smooth
monodromy gauge if and only if $C_j=0$ (their Eqs.~(38)--(40) and Appendix~B).
Third, they argue, by a Hall-response analogy with Chern insulators in a weak
electric field and by inverse-participation-ratio numerics, that nonzero $C_j$ goes
along with delocalization of quasi-energy states on the frequency lattice and with
quantized energy pumping between the drives. These are physical arguments and
numerics, presented as such: no theorem-level statement is made, the analysis is
confined to the smooth/localized dichotomy, and merely measurable quasi-energy
states are not considered. Theorem~B gives a related rigorous obstruction in the
dominated analytic Diophantine setting: a nonzero Chern class of the carrying line
excludes measurable eigensections. The elliptic unitary regime considered in those works
lies outside the unconditional reach of the present method, although Theorem~A still
applies whenever (H2) can be verified.

\emph{Recent constructive counterparts.} Faulhuber, Petersen, van Velthoven, and
Voigtlaender construct a twelve-point counterexample
\cite{FPVV2026}. Oussa \cite{Oussa4pt2026} and Dai--Deng--Shi--Wu--Yang
\cite{DDSWY2026} independently construct four-point counterexamples using their own
computer-assisted certificates. In the latter work, an interval computation proves that
the displayed vector-Zak cocycle has a topologically trivial dominated contracting line,
and the authors construct a smooth eigensection on that line. These are existence
constructions, whereas Theorem~\ref{thm:Bex} is a criterion, so neither kind of result
implies the other. The three sources are recent unrefereed preprints, and references here
are to the cited versions.

\emph{Time-frequency motivation.} The vector-Zak reduction of \cite{FPVV2026}
(Faulhuber, Petersen, van Velthoven, Voigtlaender), which disproved the
HRT conjecture \cite{HRT1996}, produces exactly the class of bundle cocycles studied
here, over the translation by the cubic vector $\tau^*=(2^{1/3}-1,\,2^{2/3}-1)$, on a
bundle with $c_1(V)[\T^2]=1$; see Section~\ref{sec:zak}, where the present results are
specialized to that setting. We make no claim about linear independence of
time-frequency systems.

\subsection*{Organization}

Section~\ref{sec:setting} fixes conventions. Sections~\ref{sec:charge}--\ref{sec:magnetic}
prove Theorem~A. Sections~\ref{sec:strip}--\ref{sec:analytic} prove Theorem~B and its converse.
Section~\ref{sec:corollaries} proves Corollary~C, Theorem~D and the sharpness of its
arithmetic hypothesis, and presents the examples.
Section~\ref{sec:zak} records the vector-Zak specialization. The paper concludes with the
principal open problem. The closing disclosure material states the scope and limitations
of the private Lean development and contains the acknowledgments, authorship statement,
and AI-assistance disclosure.

\section{Setting and conventions}\label{sec:setting}

Throughout, $\T^2=\R^2/\Z^2$ with coordinates $w=(w_1,w_2)$, standard generators
$e_1=(1,0)$, $e_2=(0,1)$, Haar (Lebesgue) probability measure $\mu$, and orientation
class $[dw_1\wedge dw_2]$. Inner products on $\C^n$ are conjugate-linear in the
\emph{first} slot. For an invertible linear map $M$ between one-dimensional spaces we
write $\norm M$ for its norm; norms of compositions of maps between lines multiply. We
write $\co(M)=\min_{\abs v=1}\abs{Mv}$ for the co-norm (minimum modulus) of a matrix and
$\sigma_1(M)\ge\sigma_2(M)$ for singular values of $M\in\C^{2\times2}$, so
$\sigma_1=\norm M$ and $\sigma_2=\co(M)$ for invertible $M\in GL(2,\C)$. The translation
is
\[
S w=w+\tau,\qquad \tau=(\alpha,\beta),
\]
and \emph{ergodic} always means: $1,\alpha,\beta$ are rationally independent over $\Q$.
This is equivalent to ergodicity of $S$ on $(\T^2,\mu)$: if $h\in L^2(\T^2)$ is
$S$-invariant, its Fourier coefficients obey
$\hat h_\nu(\eu^{2\pi\ii\nu\cdot\tau}-1)=0$, and rational independence makes
$\eu^{2\pi\ii\nu\cdot\tau}\neq1$ for $\nu\neq0$, so $h$ is constant; the converse is
immediate from the invariant character.

\begin{definition}[sewing system; sections]\label{def:sewing}
A \emph{continuous unitary sewing system} is a pair of continuous maps
$U_1,U_2\colon\R^2\to U(n)$ satisfying the consistency relation
\begin{equation}\label{eq:K1}
U_1(w+e_2)\,U_2(w)=U_2(w+e_1)\,U_1(w),\qquad w\in\R^2. \tag{K1}
\end{equation}
It presents a Hermitian vector bundle $V=V_U\to\T^2$ of rank $n$ as the quotient of
$\R^2\times\C^n$ by the $\Z^2$-action generated by
$(w,\xi)\mapsto(w+e_j,U_j(w)\xi)$; \eqref{eq:K1} is exactly the cocycle condition for
this action, which is free, properly discontinuous, and fiberwise unitary. Continuous
(resp.\ measurable) \emph{sections} of $V$ are identified with continuous (resp.\
measurable) $F\colon\R^2\to\C^n$ satisfying
\begin{equation}\label{eq:section}
F(w+e_j)=U_j(w)F(w)\qquad(j=1,2;\ \text{everywhere, resp.\ a.e.}).
\end{equation}
Since the $U_j$ are unitary, $\abs{F(\cdot)}$ descends to $\T^2$, and
$\norm F^2:=\int_{[0,1)^2}\abs F^2\,d\mu$ defines the section Hilbert space
$\mathcal H_U$ (measurable sections with finite norm, modulo null sets), a separable
Hilbert space.
\end{definition}

\begin{definition}[cocycle; transfer operator; eigensections]\label{def:cocycle}
A \emph{continuous equivariant cocycle} is a continuous $A\colon\R^2\to GL(n,\C)$
satisfying
\begin{equation}\label{eq:K2}
A(w+e_j)\,U_j(w)=U_j(Sw)\,A(w)\qquad(j=1,2,\ w\in\R^2). \tag{K2}
\end{equation}
By \eqref{eq:K2}, $A(w+e_j)=U_j(Sw)A(w)U_j(w)^{-1}$, so $\norm{A(\cdot)}$ and
$\co(A(\cdot))$ descend to continuous positive functions on $\T^2$; in particular
$0<c_0\le\co(A)\le\norm A\le C_0<\infty$ automatically. The \emph{transfer operator}
$(T_AF)(w)=A(w-\tau)F(w-\tau)$ maps sections to sections (apply \eqref{eq:K2} at
$w-\tau$ and $S(w-\tau)=w$), and $T_AF=\lambda F$ is
equivalent to the \emph{eigensection equation}
\begin{equation}\label{eq:eigen2}
A(w)F(w)=\lambda\,F(Sw)\qquad\text{a.e.}
\end{equation}
A \emph{measurable eigensection} is a measurable section $F$, not a.e.\ zero, satisfying
\eqref{eq:eigen2} for some $\lambda\in\C$; no integrability is assumed. Since $A(w)$ is
invertible, $\lambda=0$ forces $F=0$ a.e.; all statements below therefore concern
$\lambda\in\C\setminus\{0\}$ but are stated for all $\lambda\in\C$.
\end{definition}

\begin{definition}[invariant line]\label{def:invline}
An \emph{invariant line} is a continuous map $\ell\colon\R^2\to\Proj{\C^n}$ (lines
through $0$ in $\C^n$) with
\begin{equation}\label{eq:invline}
\ell(w+e_j)=U_j(w)\,\ell(w),\qquad A(w)\,\ell(w)=\ell(Sw)\qquad(j=1,2,\ w\in\R^2).
\end{equation}
It induces a continuous line subbundle $L\subset V$ over $\T^2$. A measurable section
$F$ is \emph{carried by $\ell$} if $F(w)\in\ell(w)$ for a.e.\ $w$.
\end{definition}

Here and below $\Proj{\C^n}$ denotes the projective space of $\C^n$; for $n=2$ we write
$\CP$ and use on it the round metric normalized so that in every unitary affine chart
$[r:1]\mapsto r$,
\begin{equation}\label{eq:metric}
ds=\frac{2\,\abs{dr}}{1+\abs r^2},
\end{equation}
i.e.\ the metric of the unit round sphere: the distance between the chart center and
$[r:1]$ is $2\arctan\abs r$, the distance between orthogonal lines is $\pi$ (the
diameter), and the distance between lines at Hermitian angle $\theta\in[0,\pi/2]$ is
$2\theta$. Unitaries act by isometries.

\section{The charge of an invariant line}\label{sec:charge}

Everything in this section is topological: $A$ plays no role.

\begin{lemma}[normalization]\label{lem:norm}
Let $\sigma=(\sigma_1,\sigma_2)\colon\R^2\to S^1\times S^1$ be continuous and satisfy the
consistency relation
\begin{equation}\label{eq:sigmacons}
\sigma_1(w+e_2)\,\sigma_2(w)=\sigma_2(w+e_1)\,\sigma_1(w).
\end{equation}
Then there are a unique $m\in\Z$ and a continuous $\omega\colon\R^2\to S^1$ with
\[
\sigma_j(w)=\rho^{(m)}_j(w)\,\frac{\omega(w+e_j)}{\omega(w)},\qquad
\rho^{(m)}:=\bigl(1,\ \eu^{2\pi\ii m w_1}\bigr).
\]
Moreover $m$ is unchanged if $\sigma$ is replaced by a cohomologous pair
$\sigma_j\cdot\partial_j\mu$, $\partial_j\mu:=\mu(\cdot+e_j)/\mu$, with
$\mu\colon\R^2\to S^1$ continuous.
\end{lemma}

\begin{proof}
Since $\R^2$ is simply connected, write $\sigma_j=\eu^{2\pi\ii r_j}$ with
$r_j\colon\R^2\to\R$ continuous; a lift is unique up to an additive integer constant.
By \eqref{eq:sigmacons},
\[
n(w):=r_2(w+e_1)+r_1(w)-r_1(w+e_2)-r_2(w)\in\Z
\]
for every $w$; being continuous and integer-valued on connected $\R^2$ it is a constant
$m\in\Z$, independent of the choice of lifts (each additive constant occurs twice with
opposite signs). If $\sigma$ is replaced by $\sigma_j\,\partial_j\mu$ with
$\mu=\eu^{2\pi\ii h}$, $h$ continuous, then $r_j$ is replaced by $r_j+h(\cdot+e_j)-h$
and the four $h$-differences in $n$ telescope to zero; hence $m$ is a cohomology
invariant, and $m=0$ for any coboundary pair.

\emph{Existence of the normal form.} First kill $\sigma_1$: define
$\xi(w):=-w_1\,r_1(0,w_2)$ for $0\le w_1\le1$ and extend to all of $\R^2$ by
$\xi(w+e_1):=\xi(w)-r_1(w)$; the seam values match
($\xi(1,w_2)=-r_1(0,w_2)=\xi(0,w_2)-r_1(0,w_2)$), each extension step pastes continuous
pieces agreeing on the seam, so $\xi$ is continuous and satisfies
$\xi(w+e_1)-\xi(w)=-r_1(w)$ for all $w$. Replacing $\sigma$ by the cohomologous pair
$\sigma_j\,\partial_j\eu^{2\pi\ii\xi}$ we may assume $\sigma_1\equiv1$, $r_1\equiv0$.
Now the constancy relation reads $r_2(w+e_1)-r_2(w)=m$, so
$\eta:=r_2-m\,w_1$ is $e_1$-periodic. Define $\zeta(w):=-w_2\,\eta(w_1,0)$ for
$0\le w_2\le1$ and extend by $\zeta(w+e_2):=\zeta(w)-\eta(w)$; as before $\zeta$ is
continuous, satisfies $\zeta(w+e_2)-\zeta(w)=-\eta(w)$, and is $e_1$-periodic because
$\eta$ is. Replacing $\sigma$ by $\sigma_j\,\partial_j\eu^{2\pi\ii\zeta}$ leaves
$\sigma_1\equiv1$ untouched ($\zeta$ is $e_1$-periodic) and replaces
$\sigma_2=\eu^{2\pi\ii r_2}$ by $\eu^{2\pi\ii(r_2-\eta)}=\eu^{2\pi\ii m w_1}$. Undoing
the two substitutions, $\sigma_j=\rho^{(m)}_j\,\partial_j\omega$ with
$\omega=\eu^{-2\pi\ii(\xi+\zeta)}$.

\emph{Uniqueness of $m$.} If $\rho^{(m)}_j=\rho^{(m')}_j\,\partial_j\omega$, the $j=1$
relation says $\omega$ is $e_1$-periodic, so the winding number $N(w_2)\in\Z$ of the
loop $w_1\mapsto\omega(w_1,w_2)$, $w_1\in[0,1]$, is defined; it is continuous in $w_2$,
hence constant. The $j=2$ relation reads
$\omega(w+e_2)=\eu^{2\pi\ii(m-m')w_1}\,\omega(w)$; comparing windings of the two sides
gives $N=(m-m')+N$, so $m=m'$.
\end{proof}

\begin{lemma}[normalized unit section; charge]\label{lem:charge}
Let $\ell$ be as in Definition~\ref{def:invline} (only the sewing relation in
\eqref{eq:invline} is used). Then there exist a continuous unit section
$e\colon\R^2\to\C^n$, $\abs{e}=1$, $e(w)\in\ell(w)$, and a unique integer $m=m(\ell)$,
the \emph{charge}, with
\begin{equation}\label{eq:K3}
e(w+e_1)=U_1(w)\,e(w),\qquad
e(w+e_2)=\eu^{2\pi\ii m w_1}\,U_2(w)\,e(w). \tag{K3}
\end{equation}
Any two continuous unit sections satisfying \eqref{eq:K3} have the same $m$, and differ
by a continuous $\Z^2$-periodic $S^1$-valued factor.
\end{lemma}

\begin{proof}
The pullback under $\ell\colon\R^2\to\mathbb P(\mathbb C^n)$ of the tautological line
bundle is a line bundle over the contractible paracompact base $\R^2$, hence trivial
\cite[Cor.~1.8]{HatcherVB}; normalizing a trivializing section fiberwise gives a
continuous unit section $e_0$ of $\ell$. For $j=1,2$, both $e_0(w+e_j)$ and
$U_j(w)e_0(w)$ are unit vectors in the line $\ell(w+e_j)$, so
$e_0(w+e_j)=\sigma_j(w)U_j(w)e_0(w)$ defines continuous
$\sigma_j(w)=\ip{U_j(w)e_0(w)}{e_0(w+e_j)}\in S^1$. Computing $e_0(w+e_1+e_2)$ along the
two edge orders and cancelling the common unit vector using \eqref{eq:K1} yields the
consistency \eqref{eq:sigmacons}. Apply Lemma~\ref{lem:norm} and set
$e:=\overline\omega\,e_0\;(=\omega^{-1}e_0)$; then $e$ satisfies \eqref{eq:K3} with the
lemma's $m$.
If $e,e'$ both satisfy \eqref{eq:K3} with integers $m,m'$, then $e'=\nu e$ with
$\nu\colon\R^2\to S^1$ continuous, and comparing automorphies gives
$\nu(w+e_1)=\nu(w)$ and $\nu(w+e_2)=\eu^{2\pi\ii(m'-m)w_1}\nu(w)$; the winding argument
of Lemma~\ref{lem:norm} forces $m'=m$ and then $\nu$ is $\Z^2$-periodic.
\end{proof}

\begin{lemma}[charge equals degree]\label{lem:degree}
Let $\ell,e,m$ be as in Lemma~\ref{lem:charge} and let $L\subset V$ be the induced line
subbundle over $\T^2$. Then
\[
c_1(L)\,[\T^2]=m
\]
with respect to the orientation $[dw_1\wedge dw_2]$ and the Chern--Weil normalization
$c_1(\nabla)=\tfrac{\ii}{2\pi}F_\nabla$. In particular $m=0$ iff $L$ is topologically
trivial.
\end{lemma}

\begin{proof}
Sections of $L$ correspond to sections $F$ of $V$ with $F\in\ell$; writing $F=fe$ with
$f=\ip{e}{F}$, the sewing \eqref{eq:section} and \eqref{eq:K3} translate (using
unitarity of $U_j$ and conjugate-linearity in the first slot) into
\begin{equation}\label{eq:C1}
f(w+e_1)=f(w),\qquad f(w+e_2)=\eu^{-2\pi\ii m w_1}f(w).
\end{equation}
Thus $F\mapsto f$ is an isomorphism of $L$ with the smooth model line bundle $L_m$ over
$\T^2$ whose sections are functions with the automorphy \eqref{eq:C1}, i.e.\ the
quotient of $\R^2\times\C$ by $(w,\xi)\sim(w+e_1,\xi)\sim\bigl(w+e_2,
\eu^{-2\pi\ii m w_1}\xi\bigr)$. On $L_m$ the formula
$\nabla f:=df+2\pi\ii m\,w_2\,dw_1\cdot f$ defines a Hermitian connection: the
connection form is imaginary, is $e_1$-periodic, and under $w\mapsto w+e_2$ one checks
$(\nabla f)(w+e_2)=\eu^{-2\pi\ii m w_1}(\nabla f)(w)$, so $\nabla$ descends. Its
curvature is $F_\nabla=2\pi\ii m\,dw_2\wedge dw_1$, whence
\[
\int_{\T^2}\frac{\ii}{2\pi}F_\nabla
=\int_{\T^2} m\,dw_1\wedge dw_2=m .
\]
By the Chern--Weil theorem for Hermitian line bundles, the de Rham class
\[
\left[\frac{\ii}{2\pi}F_\nabla\right]
\]
is the real image of the topological first Chern class of $L_m$
\cite[Sec.~3, Example~3]{Tamvakis2004}. Naturality under the established
bundle isomorphism $L\cong L_m$ identifies this class with the real image of $c_1(L)$.
Consequently $c_1(L)[\T^2]=m$. Since continuous complex line bundles over $\T^2$ are
classified by $c_1\in H^2(\T^2;\Z)\cong\Z$ \cite{HatcherVB,Husemoller}, $m=0$ if and
only if $L$ is topologically trivial.
\end{proof}

\begin{remark}
Only the biconditional $m\neq0\iff c_1(L)\neq0$ is consumed by the proofs below.
Nonvanishing of the Chern number is orientation-independent. Reversing only the
fundamental class changes the sign of the displayed Chern number, while an
orientation-reversing change of the ordered lattice generators also changes the
corresponding charge convention.
\end{remark}

\section{Scalarization and the phase equation}\label{sec:scalar}

From now on $\ell$ is an invariant line (Definition~\ref{def:invline}) with normalized
unit section $e$ and charge $m$.

\begin{lemma}[multiplier]\label{lem:mult}
Define the \emph{multiplier} of $\ell$ in the gauge $e$ by
\[
q(w):=\ip{e(Sw)}{A(w)e(w)},\qquad\text{so that}\qquad A(w)e(w)=q(w)\,e(Sw).
\]
Then $q$ is a continuous zero-free function, and
\begin{equation}\label{eq:qquasi}
q(w+e_1)=q(w),\qquad q(w+e_2)=\eu^{-2\pi\ii m\alpha}\,q(w).
\end{equation}
Consequently $\abs q$ is $\Z^2$-periodic, and
\[
p(w):=\frac{q(w)}{\abs{q(w)}}\;\eu^{2\pi\ii m\alpha w_2}
\]
is continuous, $\Z^2$-periodic and unimodular. Let $k=(k_1,k_2)\in\Z^2$ be its winding
vector and $g\colon\T^2\to\R$ a continuous function with
$p=\eu^{2\pi\ii(k\cdot w+g(w))}$ ($g$ exists since $p\,\eu^{-2\pi\ii k\cdot w}$ has zero
winding along both generators, hence lifts through the exponential
\cite[Prop.~1.33]{HatcherAT}; it is unique up to an additive integer).
\end{lemma}

\begin{proof}
$A(w)e(w)$ is a nonzero vector in $A(w)\ell(w)=\ell(Sw)$ and $e(Sw)$ spans that line, so
$q(w)\neq0$; continuity is clear, and $\abs{q(w)}=\abs{A(w)e(w)}$. For the
quasi-periodicity, apply \eqref{eq:K3}, \eqref{eq:K2} and the definition three times:
\begin{align*}
A(w+e_j)\,e(w+e_j)&=\sigma_j(w)\,A(w+e_j)U_j(w)e(w)
=\sigma_j(w)\,U_j(Sw)A(w)e(w)\\
&=\sigma_j(w)\,q(w)\,U_j(Sw)e(Sw),
\end{align*}
where $\sigma_1\equiv1$, $\sigma_2(w)=\eu^{2\pi\ii m w_1}$, while
$U_j(Sw)e(Sw)=\sigma_j(Sw)^{-1}e(Sw+e_j)$ and, by definition,
$A(w+e_j)\,e(w+e_j)=q(w+e_j)\,e(Sw+e_j)$. Cancelling the nonzero vector
$e(Sw+e_j)$,
\[
q(w+e_j)=\sigma_j(w)\,\sigma_j(Sw)^{-1}\,q(w),
\]
which is \eqref{eq:qquasi}: for $j=1$ trivially, and for $j=2$ because
$\sigma_2(w)/\sigma_2(Sw)=\eu^{-2\pi\ii m\alpha}$.
Periodicity of $p$: in the $e_1$-direction both factors are
invariant; in the $e_2$-direction the two factors change by $\eu^{-2\pi\ii m\alpha}$ and
$\eu^{+2\pi\ii m\alpha}$.
\end{proof}

\begin{lemma}[scalarization]\label{lem:scalar}
Assume $S$ ergodic. Let $F$ be a measurable section carried by $\ell$ satisfying
\eqref{eq:eigen2} with $\lambda\neq0$ and $F\neq0$ on a set of positive measure. Put
$f:=\ip{e}{F}$ (measurable). Then $F=fe$ a.e., and:
\begin{enumerate}
\item[\textup{(C1)}] $f(w+e_1)=f(w)$ and $f(w+e_2)=\eu^{-2\pi\ii m w_1}f(w)$ a.e.;
in particular $\abs f=\abs F$ descends to a measurable function on $\T^2$;
\item[\textup{(C2)}] $q(w)f(w)=\lambda f(Sw)$ a.e.;
\item[\textup{(C3)}] $f\neq0$ a.e.
\end{enumerate}
\end{lemma}

\begin{proof}
$F(w)\in\ell(w)$ a.e.\ and $\abs e=1$ give $F=\ip{e}{F}e=fe$ a.e., with
$\abs F=\abs f$. (C1): using \eqref{eq:section} and \eqref{eq:K3}, for a.e.\ $w$,
\[
f(w+e_j)=\ip{\sigma_j(w)U_j(w)e(w)}{U_j(w)F(w)}
=\overline{\sigma_j(w)}\,f(w)=\sigma_j(w)^{-1}f(w),
\]
by unitarity of $U_j$ and conjugate-linearity in the \emph{first} slot, the middle
expression carrying a complex conjugate $\overline{\sigma_j(w)}$ and the last equality
holding because $\abs{\sigma_j}=1$. (C2): insert $F=fe$
into \eqref{eq:eigen2} and use $A e=q\,e\circ S$:
$q(w)f(w)\,e(Sw)=\lambda f(Sw)\,e(Sw)$; cancel the unit vector. (C3): by (C1) the set
$\{f=0\}$ is $\Z^2$-invariant modulo null, hence descends to a measurable subset of
$\T^2$; by (C2) with $q\neq0\neq\lambda$ it is $S$-invariant modulo null; ergodicity
gives it measure $0$ or $1$, and measure $1$ is excluded because
$\mu\{F\neq0\}=\mu\{f\neq0\}>0$.
\end{proof}

\begin{lemma}[descalarization]\label{lem:desc}
Conversely, let $\ell$ be a continuous invariant line with normalized unit section $e$ of
charge $m$ and multiplier $q$, let $\lambda\in\C$, and let $f\colon\R^2\to\C$ be measurable
and satisfy \textup{(C1)} and \textup{(C2)} of Lemma~\ref{lem:scalar}. Then $F:=fe$ is a
measurable section of $V$ carried by $\ell$ with $A(w)F(w)=\lambda F(Sw)$ a.e.; moreover
$\abs F=\abs f$, so $F$ is a nonzero measurable eigensection as soon as $f\neq0$ on a set
of positive measure, and $\norm F=\bigl(\int_{[0,1)^2}\abs f^2\bigr)^{1/2}$.
\end{lemma}

\begin{proof}
For $j=1$, (C1) and \eqref{eq:K3} give $F(w+e_1)=f(w)U_1(w)e(w)=U_1(w)F(w)$. For $j=2$,
\[
F(w+e_2)=f(w+e_2)\,e(w+e_2)
=\eu^{-2\pi\ii m w_1}f(w)\cdot\eu^{+2\pi\ii m w_1}U_2(w)e(w)=U_2(w)F(w),
\]
the two charge factors being inverse to one another --- the opposite signs in (C1) and
\eqref{eq:K3} are produced by the conjugate-linearity of the inner product in its first
slot, exactly as in the computation in the proof of Lemma~\ref{lem:scalar}. So $F$
satisfies \eqref{eq:section}. Inserting $F=fe$ and using $Ae=q\,(e\circ S)$ and (C2),
\[
A(w)F(w)=f(w)q(w)\,e(Sw)=\lambda f(Sw)\,e(Sw)=\lambda F(Sw).
\]
Finally $\abs e=1$ gives $\abs F=\abs f$. Neither ergodicity, nor $\lambda\neq0$, nor any
integrability is used.
\end{proof}

\begin{definition}[phase-coboundary hypothesis]\label{def:H2}
With $q,p,k,g$ as in Lemma~\ref{lem:mult}, we say that $(A,\ell)$ satisfies
\emph{hypothesis \textup{(H2)}} if there is a \emph{measurable} $v\colon\T^2\to\R$ with
\begin{equation}\label{eq:H2}
v(w)-v(Sw)=g(w)-\langle g\rangle\qquad\text{for a.e.\ }w,
\qquad \langle g\rangle:=\int_{\T^2}g\,d\mu .
\end{equation}
\end{definition}

\begin{lemma}[\textup{(H2)} is well posed]\label{lem:gauge}
The validity of \textup{(H2)} does not depend on the choice of normalized unit section
$e$ in Lemma~\ref{lem:charge}.
\end{lemma}

\begin{proof}
By Lemma~\ref{lem:charge}, another normalized unit section is $e'=\nu e$ with $\nu$
continuous, $\Z^2$-periodic, unimodular; write $\nu=\eu^{2\pi\ii(\nu'\cdot w+h(w))}$
with $\nu'\in\Z^2$ its winding vector and $h\colon\T^2\to\R$ continuous. From
$Ae'=\nu\,q\,(e\circ S)=\bigl(\nu/\nu\circ S\bigr)\,q\;e'\circ S$ we get
$q'=q\cdot\nu/(\nu\circ S)$ and hence
$p'=p\,\eu^{-2\pi\ii\nu'\cdot\tau}\,\eu^{2\pi\ii(h-h\circ S)}$: the winding $k$ is
unchanged and $g'=g+(h-h\circ S)+\mathrm{const}$ modulo an integer. If $v$ solves
\eqref{eq:H2} for $g$, then $v':=v+h$ solves it for $g'$: indeed
$\langle h-h\circ S\rangle=0$, so $\langle g'\rangle=\langle g\rangle+\mathrm{const}$,
and
$v'(w)-v'(Sw)=\bigl(v(w)-v(Sw)\bigr)+\bigl(h(w)-h(Sw)\bigr)
=g'(w)-\langle g'\rangle$.
\end{proof}

\begin{lemma}[phase reduction]\label{lem:phase}
Assume the setting of Lemma~\ref{lem:scalar} and \textup{(H2)}. Let $\varphi:=f/\abs f$
(defined a.e.\ by \textup{(C3)}) and $\varphi':=\varphi\,\eu^{2\pi\ii v}$ with $v$ from
\eqref{eq:H2}. Then $\varphi'$ is measurable, $\abs{\varphi'}=1$ a.e., $\varphi'$
satisfies the automorphy \textup{(C1)}, and
\begin{equation}\label{eq:walk}
\varphi'(w+\tau)=c\;\eu^{2\pi\ii k\cdot w}\;\eu^{-2\pi\ii m\alpha w_2}\;\varphi'(w)
\qquad\text{a.e.},\qquad c=\Bigl(\tfrac{\lambda}{\abs\lambda}\Bigr)^{-1}
\eu^{2\pi\ii\langle g\rangle}\in S^1 .
\end{equation}
\end{lemma}

\begin{proof}
Taking moduli in (C2), $\abs q\,\abs f=\abs\lambda\,\abs{f\circ S}$ a.e.; dividing (C2)
by this identity (all factors nonzero a.e.),
\[
\varphi(Sw)=\Bigl(\tfrac{\lambda}{\abs\lambda}\Bigr)^{-1}
\frac{q(w)}{\abs{q(w)}}\,\varphi(w)
=\Bigl(\tfrac{\lambda}{\abs\lambda}\Bigr)^{-1}
\eu^{2\pi\ii(k\cdot w+g(w))}\,\eu^{-2\pi\ii m\alpha w_2}\,\varphi(w)\quad\text{a.e.}
\]
No modulus solvability is used; the modulus identity divides out exactly. $\varphi$
inherits (C1) because $\abs f$ is $\Z^2$-periodic. Multiplying by
$\eu^{2\pi\ii v(Sw)}$ and using \eqref{eq:H2} in the form
$v(Sw)=v(w)-g(w)+\langle g\rangle$ (a.e.)\ gives \eqref{eq:walk}; the factor
$\eu^{2\pi\ii v}$ is $\Z^2$-periodic and unimodular, so $\varphi'$ retains (C1) and
$\abs{\varphi'}=1$ a.e.
\end{proof}

\begin{lemma}[the charge-free case: a global logarithm]\label{lem:log}
Suppose $m=0$. Then $q$ is $\Z^2$-periodic, i.e.\ a continuous zero-free function on
$\T^2$, and $p=q/\abs q$. Moreover $q$ admits a continuous logarithm on $\T^2$ if and only
if $k=0$; and if $q$ is in addition real analytic and $k=0$, then
\[
h:=\log\abs q+2\pi\ii g
\]
is real analytic on $\T^2$, satisfies $\eu^{h}=q$, and is unique modulo $2\pi\ii\Z$; hence
$\langle h\rangle:=\int_{\T^2}h\,d\mu$ is well defined modulo $2\pi\ii\Z$ and
$\eu^{\langle h\rangle}$ is well defined.
\end{lemma}

\begin{proof}
Periodicity is \eqref{eq:qquasi} at $m=0$, and the extra factor $\eu^{2\pi\ii m\alpha w_2}$
in the definition of $p$ (Lemma~\ref{lem:mult}) is then trivial. Since $\T^2$ is a
$K(\Z^2,1)$ and $\C^\times$ is homotopy equivalent to $S^1=K(\Z,1)$,
$[\T^2,\C^\times]\cong\mathrm{Hom}(\Z^2,\Z)\cong\Z^2$, the isomorphism sending $q$ to the
winding vector of $q/\abs q=p$, i.e.\ to $k$. A continuous logarithm is precisely a lift of
$q$ through the covering $\exp\colon\C\to\C^\times$, whose total space is simply connected,
so by the lifting criterion \cite[Prop.~1.33]{HatcherAT}---the same one used in
Lemma~\ref{lem:mult}---such a lift exists iff $q_*=0$ on $\pi_1$, iff $k=0$. Any lift $h$
has $\operatorname{Re}h=\log\abs q$ and $\eu^{2\pi\ii(\operatorname{Im}h/2\pi)}=p$, so
$\operatorname{Im}h/2\pi$ is a continuous zero-winding phase of $p$ and therefore equals
$g$ modulo an integer. As $\exp$ is a local biholomorphism and $q$ is real analytic and
zero free, $h$ is locally a holomorphic branch of $\log$ composed with a real-analytic
map, hence real analytic; being globally defined and continuous, it is real analytic on
$\T^2$. Two lifts differ by a continuous $2\pi\ii\Z$-valued function on a connected space,
hence by a constant.
\end{proof}

\begin{remark}[the two invariants, read together]\label{rem:twoinv}
Lemma~\ref{lem:log} is the reason exactly two invariants govern this problem. The charge $m$
is the obstruction to trivializing $L$; once $L$ is trivialized the isomorphism
$L\to S^*L$ induced by $A$ becomes a function $q\colon\T^2\to\C^\times$, whose class
$k\in H^1(\T^2;\Z)$ is intrinsic --- changing the trivialization multiplies $q$ by
$\nu/(\nu\circ S)$, of winding $[\nu]-S^*[\nu]=0$ because $S$ is homotopic to the identity,
which is the content of Lemma~\ref{lem:gauge}. After both vanish, only an analytic
small-divisor problem remains; that is what Section~\ref{sec:analytic} exploits in both
directions.
\end{remark}

\section{Magnetic rigidity and the proof of Theorem A}\label{sec:magnetic}

\begin{definition}[charge space; magnetic translations]\label{def:Hm}
For $m\in\Z$ let $H_m$ be the space of (classes of) measurable $f\colon\R^2\to\C$ with
\[
f(w+e_1)=f(w),\qquad f(w+e_2)=\eu^{-2\pi\ii m w_1}f(w)\quad\text{a.e.},\qquad
\norm f^2:=\int_{[0,1)^2}\abs f^2\,dw<\infty .
\]
Restriction to $[0,1)^2$ is an isometry of $H_m$ onto $L^2([0,1)^2)$: the multipliers
are unimodular, and they satisfy the $\Z^2$-cocycle consistency (both routes from $w$ to
$w+e_1+e_2$ give the factor $\eu^{-2\pi\ii m w_1}$, using $\eu^{-2\pi\ii m}=1$), so
extension by the automorphy inverts the restriction. In particular $H_m$ is a separable
Hilbert space. For $s\in\R^2$ and $\nu\in\Z^2$ define on $H_m$
\[
(M_sf)(w):=\eu^{2\pi\ii m s_1w_2}\,f(w+s),\qquad
(E_\nu f)(w):=\eu^{2\pi\ii\nu\cdot w}f(w).
\]
\end{definition}

\begin{lemma}[magnetic translation identities]\label{lem:mag}
$M_s$ and $E_\nu$ are unitary on $H_m$, and for all $s,t\in\R^2$, $\nu\in\Z^2$:
\begin{equation}\label{eq:comm}
M_sM_t=\eu^{2\pi\ii m(t_1s_2-s_1t_2)}\,M_tM_s,\qquad
M_sE_\nu=\eu^{2\pi\ii\nu\cdot s}\,E_\nu M_s .
\end{equation}
\end{lemma}

\begin{proof}
$M_s$ preserves the automorphy: the $e_1$-relation is clear; for $e_2$,
\[
\begin{aligned}
(M_sf)(w+e_2)&=\eu^{2\pi\ii m s_1(w_2+1)}f(w+s+e_2)\\
&=\eu^{2\pi\ii m s_1w_2}\,\eu^{2\pi\ii m s_1}\,\eu^{-2\pi\ii m(w_1+s_1)}f(w+s)
=\eu^{-2\pi\ii m w_1}(M_sf)(w).
\end{aligned}
\]
Unitarity: the prefactor is unimodular, $\abs f^2$ is $\Z^2$-periodic, and Lebesgue
measure on the torus is translation-invariant. $E_\nu$ is a unimodular periodic
multiplier. For \eqref{eq:comm}:
\begin{align*}
(M_sM_tf)(w)&=\eu^{2\pi\ii m s_1w_2}\,\eu^{2\pi\ii m t_1(w_2+s_2)}f(w+s+t),\\
(M_tM_sf)(w)&=\eu^{2\pi\ii m t_1w_2}\,\eu^{2\pi\ii m s_1(w_2+t_2)}f(w+t+s);
\end{align*}
the ratio of prefactors is $\eu^{2\pi\ii m(t_1s_2-s_1t_2)}$. Similarly
$(M_sE_\nu f)(w)=\eu^{2\pi\ii m s_1w_2}\eu^{2\pi\ii\nu\cdot(w+s)}f(w+s)
=\eu^{2\pi\ii\nu\cdot s}(E_\nu M_sf)(w)$.
\end{proof}

\begin{lemma}[magnetic rigidity]\label{lem:rigid}
Let $m\in\Z\setminus\{0\}$, $k\in\Z^2$, $c\in S^1$, and $\tau=(\alpha,\beta)$ with
$\beta\notin\Q$ \emph{or} $\alpha\notin\Q$. Then the unitary operator
\[
W:=\bar c\,E_{-k}\,M_\tau\ \colon\ H_m\to H_m
\]
has no nonzero fixed vector. More precisely, for every $s\in\R^2$,
\begin{equation}\label{eq:WM}
W M_s=\chi(s)\,M_sW,\qquad
\chi(s)=\eu^{2\pi\ii\left[(k_1+m\beta)s_1+(k_2-m\alpha)s_2\right]},
\end{equation}
and if $W\psi=\psi$ with $\norm\psi\neq0$, the family
$\{M_s\psi\}$ contains uncountably many pairwise orthogonal vectors of equal nonzero
norm, contradicting separability of $H_m$.
\end{lemma}

\begin{proof}
For \eqref{eq:WM}, apply \eqref{eq:comm} to the pair $(\tau,s)$:
$M_\tau M_s=\eu^{2\pi\ii m(s_1\beta-\alpha s_2)}M_sM_\tau$, and
$E_{-k}M_s=\eu^{2\pi\ii k\cdot s}M_sE_{-k}$; multiplying by $\bar c$ and composing gives
\eqref{eq:WM}. Suppose $W\psi=\psi$, $\psi\neq0$. Then for every $s$,
$W(M_s\psi)=\chi(s)M_sW\psi=\chi(s)(M_s\psi)$ and $\norm{M_s\psi}=\norm\psi\neq0$:
each $M_s\psi$ is an eigenvector of the unitary $W$ with eigenvalue $\chi(s)$.
Eigenvectors of a unitary for distinct eigenvalues are orthogonal: if $Wx=\mu x$,
$Wy=\nu y$ with $\abs\mu=\abs\nu=1$ then
$\ip xy=\ip{Wx}{Wy}=\bar\mu\nu\ip xy$. Since $m\neq0$ and $\beta\notin\Q$ (say), the
number $k_1+m\beta$ is irrational, in particular nonzero; hence
$s_1\mapsto\chi\bigl((s_1,0)\bigr)=\eu^{2\pi\ii(k_1+m\beta)s_1}$ is injective on
$[0,\delta)$ for any $0<\delta<\min\{1,\abs{k_1+m\beta}^{-1}\}$ (if
$\chi((s_1,0))=\chi((s_1',0))$ then $(k_1+m\beta)(s_1-s_1')\in\Z$, and
$\abs{(k_1+m\beta)(s_1-s_1')}<1$ forces $s_1=s_1'$). Thus
$\{M_{(s_1,0)}\psi:\ 0\le s_1<\delta\}$ is an uncountable family of pairwise orthogonal
vectors of norm $\norm\psi$. After normalization these form an uncountable orthonormal
family; the open balls of radius $\sqrt2/2$ around its members are pairwise disjoint, and
a separable metric space contains no uncountable family of pairwise disjoint open balls.
If instead $\alpha\notin\Q$, run the same argument with $s=(0,s_2)$ and
$k_2-m\alpha$. No continuity of $s\mapsto M_s\psi$ is used.
\end{proof}

\begin{theorem}[Chern-class obstruction]\label{thm:A}
Let $n\ge1$ and let $(U_1,U_2)$ be a continuous unitary sewing system with values in
$U(n)$ \textup{(Definition~\ref{def:sewing})}, $A$ a continuous equivariant cocycle
\textup{(Definition~\ref{def:cocycle})}, and assume $1,\alpha,\beta$ rationally
independent over $\Q$. Let $\ell$ be a continuous invariant line
\textup{(Definition~\ref{def:invline})} whose line bundle $L$ has $c_1(L)\neq0$, and
assume $(A,\ell)$ satisfies \textup{(H2)} \textup{(Definition~\ref{def:H2})}. Then for
every $\lambda\in\C$, every measurable section $F$ of $V$ satisfying
$A(w)F(w)=\lambda F(w+\tau)$ a.e.\ and $F(w)\in\ell(w)$ a.e.\ vanishes almost
everywhere. In particular $L$ carries no measurable eigensection, and a fortiori no
eigensection in $\mathcal H_U$.
\end{theorem}

\begin{proof}
For $\lambda=0$ invertibility of $A(w)$ gives $F=0$ a.e. Let $\lambda\neq0$ and suppose
$\mu\{F\neq0\}>0$. By Lemma~\ref{lem:charge} the line $\ell$ has a normalized unit
section $e$ with charge $m$, and $m\neq0$ by Lemma~\ref{lem:degree}. Lemmas~\ref{lem:mult},
\ref{lem:scalar} and \ref{lem:phase} produce a measurable $\varphi'$ with
$\abs{\varphi'}=1$ a.e., the automorphy (C1)---so $\varphi'\in H_m$ with
$\norm{\varphi'}=1$---and the walk equation \eqref{eq:walk}. By \eqref{eq:walk},
\[
(M_\tau\varphi')(w)=\eu^{2\pi\ii m\alpha w_2}\varphi'(w+\tau)
=c\,\eu^{2\pi\ii k\cdot w}\varphi'(w),
\]
so $E_{-k}M_\tau\varphi'=c\,\varphi'$, i.e.\ $W\varphi'=\varphi'$ for
$W=\bar cE_{-k}M_\tau$. Rational independence of $1,\alpha,\beta$ implies
$\beta\notin\Q$, so Lemma~\ref{lem:rigid} forbids a nonzero fixed vector of $W$ in
$H_m$. This contradiction proves $F=0$ a.e.
\end{proof}

\begin{remark}[what each hypothesis does]\label{rem:hyps}
Ergodicity is consumed exactly once, in (C3) ($f\neq0$ a.e.); the rigidity step needs
only $\beta\notin\Q$ (or $\alpha\notin\Q$), which follows. Hypothesis (H2) is consumed
exactly once, in Lemma~\ref{lem:phase}, and only measurable solvability is needed---%
strictly weaker than continuous solvability. Neither a Lyapunov gap, nor domination, nor
any Diophantine condition, nor membership of $F$ in $\mathcal H_U$ enters. The modulus
half of the multiplier ($\log\abs q$) is never solved for: it divides out.
\end{remark}

\begin{proposition}[the charge-zero case: classical winding obstruction]\label{prop:winding}
In the setting of Theorem~\ref{thm:A} but with $c_1(L)=0$ (i.e.\ $m=0$) and $k\neq0$,
the same conclusion holds. This statement is classical in substance---for $m=0$,
$H_0=L^2(\T^2)$, the operators $M_s$ are plain translations, and the obstruction is the
Anzai--Furstenberg-type coboundary obstruction for the winding class
\cite{Anzai1951,Furstenberg1961,GLL1991,ILR1993}---and is included only for
completeness.
\end{proposition}

\begin{proof}
Run the proof of Theorem~\ref{thm:A} verbatim; now $\chi(s)=\eu^{2\pi\ii k\cdot s}$.
If $k_1\neq0$, the family $s_1\mapsto\chi((s_1,0))$ is injective on
$[0,1/(2\abs{k_1}))$; if $k_1=0\neq k_2$ use the $s_2$-axis. The rigidity argument is
unchanged.
\end{proof}

\begin{remark}[mechanism lineage]\label{rem:lineage}
For $m\neq0$ the pair $M_{e_1},M_{e_2}$ generates a charge-$m$ magnetic translation
family in the sense of Zak \cite{Zak1964}, and $W$ lies in its translate by the
irrational vector $\tau$; Lemma~\ref{lem:rigid} is the classical
``covariance excludes eigenvectors'' argument in the Weyl-commutation lineage
\cite{vonNeumann1931,AvronHerbst1977}. All facts used are proved inline; the citations
record lineage, not logical dependence. Here
\[
\chi(s)=\eu^{2\pi\ii\left[(k_1+m\beta)s_1+(k_2-m\alpha)s_2\right]}.
\]
If either coefficient is nonzero, restriction to the corresponding real one-parameter
subgroup has image $S^1$. Thus the covariance argument fails only at the exact resonance
\[
k_1+m\beta=0,\qquad k_2-m\alpha=0.
\]
For $m\neq0$ and $k\in\Z^2$, this resonance forces $m\alpha$ and $m\beta$ to be
integers. At resonance the covariance argument gives no contradiction, but it does not
itself assert the existence of an eigensection.
\end{remark}

\subsection*{The abstract converse}

Theorem~\ref{thm:A} and Proposition~\ref{prop:winding} exhaust the mechanism: it has no
purchase when $m=0$ and $k=0$. In that case the equation can in fact always be solved,
under the exact mirror of \textup{(H2)}.

\begin{definition}[coboundary hypothesis for existence]\label{def:E2}
Assume $m=0$ and $k=0$, and let $h=\log\abs q+2\pi\ii g$ be the global branch of
Lemma~\ref{lem:log}. We say that $(A,\ell)$ satisfies \emph{hypothesis \textup{(E2)}} if
there is a \emph{measurable} $\Psi\colon\T^2\to\C$ with
\begin{equation}\label{eq:E2}
\Psi(Sw)-\Psi(w)=h(w)-\langle h\rangle\qquad\text{for a.e.\ }w .
\end{equation}
\end{definition}

\begin{theorem}[abstract existence]\label{thm:Aex}
Let $n\ge1$, let $(U_1,U_2)$ be a continuous unitary sewing system, $A$ a continuous
equivariant cocycle, and $\ell$ a continuous invariant line with $m=0$ and $k=0$ satisfying
\textup{(E2)}. Put $\lambda_0:=\eu^{\langle h\rangle}$. Then
\[
F_0:=\eu^{\Psi}\,e
\]
is a nonzero measurable eigensection of $T_A$ carried by $\ell$ with eigenvalue
$\lambda_0$. If $\Psi$ is bounded --- in particular if it is continuous --- then $F_0$ is
nowhere zero and $F_0\in\mathcal H_U$.
\end{theorem}

\begin{proof}
$f:=\eu^{\Psi}$ is measurable, zero free, and $\Z^2$-periodic because $\Psi$ is a function
on $\T^2$; since $m=0$, $\Z^2$-periodicity \emph{is} the automorphy (C1). By
\eqref{eq:E2},
\[
\frac{f(Sw)}{f(w)}=\eu^{\Psi(Sw)-\Psi(w)}=\eu^{h(w)-\langle h\rangle}
=\frac{q(w)}{\lambda_0}\qquad\text{a.e.},
\]
which is (C2). Lemma~\ref{lem:desc} turns $f$ into the eigensection $F_0=fe$. If $\Psi$ is
bounded then $\abs{F_0}=\eu^{\operatorname{Re}\Psi}$ is bounded above and below.
\end{proof}

\begin{remark}[what the two directions consume, and where $m=0$ enters]\label{rem:asym}
Splitting $h$ into real and imaginary parts, \textup{(E2)} is the conjunction of
\textup{(H2)} --- the manuscript's phase hypothesis --- \emph{and} the corresponding
statement for $\log\abs q$. This is the precise sense in which the modulus does not divide
out in the existence direction: Remark~\ref{rem:hyps} records that
Theorem~\ref{thm:A} never solves for $\log\abs q$, whereas Theorem~\ref{thm:Aex} must.
The abstract situation is therefore asymmetric --- exclusion needs \textup{(H2)} alone,
existence needs \textup{(H2)} and its modulus counterpart --- and it is only in the
analytic--Diophantine regime of Section~\ref{sec:analytic}, where
Lemma~\ref{lem:solve} supplies both at once, that $m$ and $k$ become a complete criterion.
Note also that Theorem~\ref{thm:Aex} uses ergodicity nowhere, whereas
Theorem~\ref{thm:A} consumes it once. Finally, $m=0$ is used in exactly one place, and
structurally: $\eu^{\Psi}$ is $\Z^2$-periodic, which is (C1) only when $m=0$; for
$m\neq0$, (C1) demands $f(w+e_2)=\eu^{-2\pi\ii m w_1}f(w)$, which no periodic $f$
satisfies. The construction fails exactly where Theorem~\ref{thm:A} bites.
\end{remark}

\section{Analytic cocycles: strip extension of dominated lines}\label{sec:strip}

We now prepare the discharge of (H2). For $\delta>0$ let
\[
\Omega_\delta:=\{w\in\C^2:\ \abs{\operatorname{Im}w_1}\le\delta,\
\abs{\operatorname{Im}w_2}\le\delta\},
\]
a closed tube around $\R^2$, with interior $\Omega_\delta^\circ$; $S$ acts on it by the
real translation $w\mapsto w+\tau$, and $\operatorname{Re}\colon\Omega_\delta\to\R^2$ is
the coordinatewise real part.

\begin{definition}[strip-admissible analytic data]\label{def:stripadm}
The triple $(U_1,U_2,A)$ is \emph{strip-admissible} if for some $\delta_0>0$ the maps
$U_1,U_2,A$ extend to matrix functions holomorphic on an open neighborhood of the
closed tube $\Omega_{\delta_0}$ and bounded on $\Omega_{\delta_0}$,
\[
\sup_{\Omega_{\delta_0}}\bigl(\norm{U_1}+\norm{U_2}+\norm{A}\bigr)<\infty,
\]
such that $U_j(w)\in U(n)$ and $\det A(w)\neq0$ for all real $w$, and \eqref{eq:K1},
\eqref{eq:K2} hold on $\R^2$. Holomorphy on a neighborhood of the \emph{closed} tube
is what the Cauchy estimates below consume: for $w\in\Omega_\delta$ with
$0<\delta<\delta_0$ the closed polydisc of polyradius $\delta_0-\delta$ centered at
$w$ lies in $\Omega_{\delta_0}$, and for real $w$ the closed polydisc of polyradius
$\delta_0$ does.
\end{definition}

The neighborhood formulation does not shrink the class of admissible data:
since $\delta_0$ is existentially quantified, a triple whose extensions are
holomorphic merely on the open tube $\Omega^\circ_{\delta_0}$ and bounded on
$\Omega_{\delta_0}$ satisfies Definition~\ref{def:stripadm} with $\delta_0$
replaced by $\delta_0/2$. Only the numerical values of constants tied to a
particular choice of $\delta_0$ (such as $L_A=S_0/\delta_0$ in
Lemma~\ref{lem:uc}) depend on that choice; no statement in this paper does.

\begin{lemma}[extension of the identities; uniform strip estimates]\label{lem:cauchy}
Let $(U_1,U_2,A)$ be strip-admissible. Then:
\begin{enumerate}
\item \eqref{eq:K1} and \eqref{eq:K2} hold on all of $\Omega_{\delta_0}$;
\item for every matrix function $G$ holomorphic on an open neighborhood of
$\Omega_{\delta_0}$ and bounded on $\Omega_{\delta_0}$, and every
$0<\delta<\delta_0$,
\begin{equation}\label{eq:cauchy}
\norm{G(w)-G(\operatorname{Re}w)}\ \le\ \frac{2\,\delta}{\delta_0-\delta}\,
\sup_{\Omega_{\delta_0}}\norm G\qquad(w\in\Omega_\delta),
\end{equation}
uniformly in $\operatorname{Re}w$; more generally, for the same $\delta$ and all
$w,w'\in\Omega_\delta$ with $\operatorname{Re}w=\operatorname{Re}w'$,
\begin{equation}\label{eq:cauchy2}
\norm{G(w)-G(w')}\ \le\
\frac{2\,\abs{\operatorname{Im}w-\operatorname{Im}w'}_{\infty}}{\delta_0-\delta}\,
\sup_{\Omega_{\delta_0}}\norm G ,
\end{equation}
in which the denominator retains the half-width $\delta$ of a strip containing the whole
segment $[w',w]$ and is \emph{not} replaced by the separation of the two points;
\item there are $\delta_1\in(0,\delta_0)$ and $C_U\ge1$ such that on
$\Omega_{\delta_1}$ each $U_j(w)$ is invertible with
$\norm{U_j(w)^{\pm1}}\le1+C_U\delta_1$, and
$\inf_{\Omega_{\delta_1}}\abs{\det A}>0$; moreover, for each $N$, the block product
$A^{(N)}(w):=A(S^{N-1}w)\cdots A(w)$, with the convention
$A^{(0)}(w):=I$, is holomorphic on an open neighborhood of
$\Omega_{\delta_0}$ and bounded on $\Omega_{\delta_0}$,
and satisfies the equivariance
\begin{equation}\label{eq:blockeq}
A^{(N)}(w+e_j)\,U_j(w)=U_j(S^Nw)\,A^{(N)}(w).
\end{equation}
\end{enumerate}
\end{lemma}

\begin{proof}
(1) Each identity is an equality of functions holomorphic on an open neighborhood of
$\Omega_{\delta_0}$, holding on $\R^2$. Since $\Omega_{\delta_0}$ is connected, it
lies in a single connected component $D$ of that neighborhood, a connected open set
containing $\R^2$. A holomorphic function on $D$ vanishing on the real slice vanishes
identically: at a real point all its derivatives in real
coordinate directions vanish, and by holomorphy these determine all complex partial
derivatives, so its Taylor series vanishes; connectedness of $D$ finishes.
(2) For $w\in\Omega_\delta$ the closed polydisc of polyradius $\delta_0-\delta$ around
$w$ lies in $\Omega_{\delta_0}$, so the Cauchy estimate gives
$\norm{\partial G/\partial w_i}\le\sup\norm G/(\delta_0-\delta)$ on \emph{all} of
$\Omega_\delta$. For $w,w'\in\Omega_\delta$ with $\operatorname{Re}w=\operatorname{Re}w'$
the segment $[w',w]$ lies in $\Omega_\delta$, so that bound holds along it; integrating,
\[
\norm{G(w)-G(w')}\le
\bigl(\abs{\operatorname{Im}w_1-\operatorname{Im}w_1'}
+\abs{\operatorname{Im}w_2-\operatorname{Im}w_2'}\bigr)
\frac{\sup\norm G}{\delta_0-\delta},
\]
and \eqref{eq:cauchy2} follows from $\abs\cdot_1\le2\abs\cdot_\infty$ in two variables.
Taking $w'=\operatorname{Re}w$, whose imaginary separation from $w$ is at most $\delta$,
gives \eqref{eq:cauchy}. Both estimates are uniform in $\operatorname{Re}w$.

The half-width $\delta$ in the denominator of \eqref{eq:cauchy2} cannot be replaced by
the separation $\abs{\operatorname{Im}w-\operatorname{Im}w'}_\infty$: the Cauchy bound
degrades towards the boundary of $\Omega_{\delta_0}$, and what controls it is the
distance from the connecting segment to that boundary, not the length of the segment.
With $\delta_0=1$ and $G(w)=\eu^{-\ii Tw_1}$ one has
$\sup_{\Omega_1}\abs G=\eu^{T}$, while the points $w=(\ii(1-\varepsilon),0)$ and
$w'=(\ii(1-2\varepsilon),0)$ of $\Omega_{1-\varepsilon}$ have equal real parts,
imaginary separation $\varepsilon$, and, for $T=(\log2)/\varepsilon$,
\[
\frac{\abs{G(w)-G(w')}}{\sup_{\Omega_1}\abs G}
=\eu^{-T\varepsilon}-\eu^{-2T\varepsilon}=\tfrac14 ,
\]
which exceeds $2\varepsilon/(1-\varepsilon)$ for all small $\varepsilon$.
(3) At real points $U_j$ is unitary and---as noted in
Definition~\ref{def:cocycle}---$\abs{\det A}$ is $\Z^2$-periodic and continuous, hence
bounded below by some $c_1>0$ on $\R^2$. Both statements transfer to a thin strip by
\eqref{eq:cauchy} applied to $U_j$ and to $\det A$ (a scalar holomorphic on the
neighborhood and bounded on the tube). For the block product, let $D$ be the
neighborhood of Definition~\ref{def:stripadm}: since the closed tube is
$S$-invariant, each factor $A(S^iw)$ is holomorphic on the open neighborhood
$S^{-i}D\supseteq\Omega_{\delta_0}$, so the finite product is holomorphic on
$\bigcap_{0\le i<N}S^{-i}D$, still an open neighborhood of $\Omega_{\delta_0}$, and
bounded on $\Omega_{\delta_0}$; \eqref{eq:blockeq} follows from \eqref{eq:K2} by
induction:
$A^{(N+1)}(w+e_j)U_j(w)=A(S^Nw+e_j)A^{(N)}(w+e_j)U_j(w)
=A(S^Nw+e_j)U_j(S^Nw)A^{(N)}(w)=U_j(S^{N+1}w)A^{(N+1)}(w)$, using
$S^N(w+e_j)=S^Nw+e_j$.
\end{proof}

\begin{definition}[dominated splitting]\label{def:dom}
Let $n=2$. A \emph{dominated splitting} for $(U,A)$ is a pair of invariant lines
$\ell_s\neq\ell_u$ (pointwise transversal, each as in Definition~\ref{def:invline};
write $E^s,E^u$ for the induced subbundles of $V$, so $V=E^s\oplus E^u$) such that for
some integer $N\ge1$ and some constant $0<\kappa_0<1$,
\begin{equation}\label{eq:dom}
D_N(z):=\frac{\ \norm{A^{(N)}(z)|_{E^s(z)}}\ }{\co\bigl(A^{(N)}(z)|_{E^u(z)}\bigr)}
\ \le\ \kappa_0^{2}\qquad\text{for all real }z .
\end{equation}
This is the standard domination condition for a splitting into two invariant lines;
compare the singular-value characterizations of dominated splittings in
\cite[Theorem~1.1]{BochiGourmelon2009}.
\end{definition}

\begin{lemma}[consequences of domination]\label{lem:domcons}
Assume \eqref{eq:dom}. Then:
\begin{enumerate}
\item \emph{(exact multiplicativity)} $D_{kN}(z)=\prod_{j=0}^{k-1}D_N(S^{jN}z)\le
\kappa_0^{2k}$ for all $k\ge1$ and real $z$;
\item \emph{(singular value ratio)} $\sigma_2\bigl(A^{(kN)}(z)\bigr)/
\sigma_1\bigl(A^{(kN)}(z)\bigr)\le D_{kN}(z)\le\kappa_0^{2k}$ for real $z$;
\item \emph{(Lyapunov gap)} the Haar averages
$\lambda_u:=\langle\log\abs{q_u}\rangle$ and $\lambda_s:=\langle\log\abs{q_s}\rangle$
of the one-step multipliers of $E^u,E^s$ satisfy
$\lambda_u-\lambda_s\ge\tfrac1N\log\kappa_0^{-2}>0$. Equivalently, when the base
translation is ergodic, these are the almost-everywhere Birkhoff limits.
\end{enumerate}
\end{lemma}

\begin{proof}
(1) Restrictions of the blocks to invariant lines compose as maps between
one-dimensional spaces, whose norms and co-norms multiply. (2) $\sigma_2(M)\le
\norm{M|_{E^s}}$ (evaluate the minimum over the unit sphere at a unit vector of $E^s$)
and $\sigma_1(M)\ge\co(M|_{E^u})$. (3) With $q_i(w)=\ip{e_i(Sw)}{A(w)e_i(w)}$ as in
Lemma~\ref{lem:mult} for unit sections $e_i$ of the two lines,
$\abs{q_i(w)}=\norm{A(w)|_{E^i(w)}}$; exact multiplicativity along lines gives
$N\langle\log\abs{q_s}\rangle=\langle\log\norm{A^{(N)}|_{E^s}}\rangle$ and
$N\langle\log\abs{q_u}\rangle=\langle\log\co(A^{(N)}|_{E^u})\rangle$, and
\eqref{eq:dom} bounds the difference of the integrands below by
$\log\kappa_0^{-2}$.
\end{proof}

\begin{lemma}[strip domination and singular directions]\label{lem:stripdom}
Assume strip-admissibility and \eqref{eq:dom}. Set
$\varepsilon_1:=\min_{z}\operatorname{dist}\bigl(E^s(z),E^u(z)\bigr)>0$ (spherical
metric \eqref{eq:metric}; the minimum over $\T^2$ is positive by continuity,
transversality and compactness). For every $k\in\N$ with $\kappa:=2\kappa_0^{2k}<1$
(all sufficiently large $k$) there is $\delta(k)>0$ such that,
with $\mathcal N:=kN$ and $M(w):=A^{(\mathcal N)}(w)$, the following hold on
$\Omega_{\delta(k)}$:
\begin{enumerate}
\item $0<s_-\le\sigma_2(M(w))\le\sigma_1(M(w))\le s_+<\infty$ and
$\eta(w):=\sigma_2(M(w))/\sigma_1(M(w))\le\kappa$;
\item the right singular line $v(w)$ (for $\sigma_2$) and the left singular line
$u(w)$ (for $\sigma_1$) are defined and continuous, with modulus of continuity in the
$\operatorname{Im}$-direction uniform in $\operatorname{Re}w$;
\item at real $z$:
$\operatorname{dist}\bigl(v(z),E^s(z)\bigr)\le\pi\kappa_0^{2k}$ and
$\operatorname{dist}\bigl(u(z),E^u(S^{\mathcal N}z)\bigr)\le\pi\kappa_0^{2k}$.
\end{enumerate}
\end{lemma}

\begin{proof}
(1) At real $z$, parts (1)--(2) of Lemma~\ref{lem:domcons} give
$\eta(z)\le\kappa_0^{2k}$, while $\sigma_2\ge c_0^{\mathcal N}$ and
$\sigma_1\le C_0^{\mathcal N}$ by Definition~\ref{def:cocycle}. Singular values are
$1$-Lipschitz in the operator norm: $\sigma_1=\norm\cdot$ satisfies the triangle
inequality, and
$\abs{\sigma_2(X)-\sigma_2(Y)}\le\norm{X-Y}$ because
$\sigma_2(X)=\min_{\abs v=1}\abs{Xv}$. By Lemma~\ref{lem:cauchy}(2) applied to
$M$, choosing $\delta(k)$ small makes the perturbation smaller than
$\min\{s_-/4,\ \kappa_0^{2k}s_-/4\}$ with $s_-:=c_0^{\mathcal N}/2$, which gives (1)
with the factor $2$ absorbing the perturbation.
(2) $\sigma_1^2,\sigma_2^2$ are the eigenvalues of the Hermitian
$H(w):=M(w)^*M(w)$, separated by
$\sigma_1^2-\sigma_2^2\ge(1-\kappa^2)s_-^2=:\gamma>0$ on the strip. Let $P(w)$ be the
spectral projection of $H(w)$ for the eigenvalue $\sigma_2^2$, given by the Riesz
integral over the circle $\Gamma$ of radius $\gamma/2$ centered at
$\sigma_2(M(w))^2$; for Hermitian matrices
$\norm{(\zeta-H)^{-1}}=\operatorname{dist}(\zeta,\operatorname{spec}H)^{-1}\le2/\gamma$
on $\Gamma$. If $\norm{H(w)-H(w')}\le\gamma/4$, the spectrum of $H(w')$ moves by at
most $\gamma/4$, both resolvents are bounded by $4/\gamma$ on $\Gamma$, and the second
resolvent identity gives $\norm{P(w)-P(w')}\le(4/\gamma)\norm{H(w)-H(w')}$; pairs with
$\norm{\Delta H}\ge\gamma/4$ satisfy the same bound trivially since
$\norm{\Delta P}\le1$. The modulus is explicit. Put
$B:=\sup_{\Omega_{\delta_0}}\norm M$. For $w,w'\in\Omega_{\delta}$ with
$\operatorname{Re}w=\operatorname{Re}w'$ and
$\eta:=\abs{\operatorname{Im}w-\operatorname{Im}w'}_\infty$, the identity
$H(w)-H(w')=M(w)^*\bigl(M(w)-M(w')\bigr)+\bigl(M(w)-M(w')\bigr)^*M(w')$ and
\eqref{eq:cauchy2} give
\begin{equation}\label{eq:Hmod}
\norm{H(w)-H(w')}\ \le\ 2B\,\norm{M(w)-M(w')}\ \le\ \frac{4B^{2}\eta}{\delta_0-\delta},
\qquad
\norm{P(w)-P(w')}\ \le\ \frac{16B^{2}\eta}{\gamma\,(\delta_0-\delta)} .
\end{equation}
Note that the denominator is fixed by the strip $\Omega_\delta$ containing both points,
not by $\eta$; this is the form of Lemma~\ref{lem:cauchy}(2) that the sequel consumes,
and it is applied only to pairs $(w,\operatorname{Re}w)$, for which $\eta\le\delta$.
Hence $P$ and $v=\operatorname{range}P$ have a modulus of continuity in the
$\operatorname{Im}$-direction uniform in $\operatorname{Re}w$; note that for rank-one orthogonal projections
$\operatorname{dist}(\operatorname{range}P,\operatorname{range}P')\le
\pi\norm{P-P'}$ in the metric \eqref{eq:metric}. The line $u$ is treated the same way
via $MM^*$.
(3) Let $z$ be real, $e^s$ a unit vector of $E^s(z)$, and expand
$e^s=c_1v_1+c_2v_2$ in the right singular basis of $M=M(z)$. Then
$\norm{Me^s}^2=\abs{c_1}^2\sigma_1^2+\abs{c_2}^2\sigma_2^2\ge\abs{c_1}^2\sigma_1^2$,
while
$\norm{Me^s}=\norm{M|_{E^s}}\le D_{\mathcal N}(z)\,\co(M|_{E^u})\le
D_{\mathcal N}(z)\,\sigma_1$; hence $\abs{c_1}\le D_{\mathcal N}(z)\le\kappa_0^{2k}$.
The Hermitian angle $\theta$ between $E^s(z)$ and $[v_2]$ has $\sin\theta=\abs{c_1}$,
and $\operatorname{dist}=2\theta\le\pi\sin\theta$ on $[0,\pi/2]$. For the second
estimate apply the same computation to $M^{-1}=V\Sigma^{-1}U^*$ and the
$M^{-1}$-invariant line $E^u(S^{\mathcal N}z)$: the role of $[v_2]$ is played by the
top left singular line $u(z)$ of $M$, and the relevant ratio equals
$\norm{M^{-1}|_{E^u(S^{\mathcal N}z)}}\,/\,\co(M^{-1}|_{E^s(S^{\mathcal N}z)})
=\norm{M|_{E^s(z)}}/\co(M|_{E^u(z)})=D_{\mathcal N}(z)$, because norms of inverses of
maps between lines are reciprocals.
\end{proof}

\begin{lemma}[M\"obius contraction]\label{lem:mobius}
Let $M\in GL(2,\C)$ with singular values $\sigma_1\ge\sigma_2>0$,
$\eta=\sigma_2/\sigma_1$, and singular frames $Mv_i=\sigma_iu_i$. For
$0<\varepsilon\le1$ let
$K_\varepsilon:=\{\ell\in\CP:\operatorname{dist}(\ell,[u_1])\ge\varepsilon\}$. Then, in
the metric \eqref{eq:metric}:
\begin{enumerate}
\item $M^{-1}\bigl(K_\varepsilon\bigr)\subseteq
\{\ell:\operatorname{dist}(\ell,[v_2])\le4\eta/\varepsilon\}$;
\item if $K\subseteq K_\varepsilon$ is geodesically convex, then
$\operatorname{dist}\bigl(M^{-1}\ell,M^{-1}\ell'\bigr)\le
\eta\Bigl(1+\tfrac{4}{\varepsilon^2}\Bigr)\operatorname{dist}(\ell,\ell')$ for
$\ell,\ell'\in K$.
\end{enumerate}
Moreover every spherical disc of radius $<\pi/2$ is geodesically convex.
\end{lemma}

\begin{proof}
Unitaries are isometries, so we may compute with $\Sigma^{-1}$ in the chart
$[r:1]:=[ru_1+u_2]$ on the input side and $[r':1]:=[r'v_1+v_2]$ on the output side:
$M^{-1}(ru_1+u_2)=\sigma_2^{-1}(\eta r\,v_1+v_2)$, i.e.\ the chart action is
$r\mapsto\eta r$. Distances in such a chart:
$\operatorname{dist}([r:1],[0:1])=2\arctan\abs r\le2\abs r$ and
$\operatorname{dist}([r:1],[\infty])=2\arctan(1/\abs r)\le2/\abs r$.
(1) $\ell\in K_\varepsilon$ gives $\varepsilon\le2/\abs r$, so $\abs r\le2/\varepsilon$;
the output coordinate is $\eta\abs r\le2\eta/\varepsilon$, so
$\operatorname{dist}(M^{-1}\ell,[v_2])\le4\eta/\varepsilon$.
(2) Along the geodesic joining $\ell,\ell'$, which stays in $K\subseteq K_\varepsilon$,
every point has chart coordinate $\abs r\le R:=2/\varepsilon$; the map $r\mapsto\eta r$
scales the metric density \eqref{eq:metric} by
$\eta(1+\abs r^2)/(1+\eta^2\abs r^2)\le\eta(1+R^2)$ there; integrating along the
geodesic bounds the length of the image path, hence the distance of the image points.
Convexity of discs of radius $r_0<\pi/2$: such a disc lies in the open hemisphere
around its center; gnomonic (central) projection from the center maps the hemisphere to
the plane, great-circle arcs to straight segments, and the disc to a Euclidean disc,
so the minimizing arc between two disc points is the preimage of a segment and stays in
the disc.
\end{proof}

The next lemma supplies the one uniformity that the holomorphy argument of
Theorem~\ref{thm:stripline} consumes. It concerns the \emph{lifted} stable line field
$\ell_s\colon\R^2\to\CP$: this satisfies the sewing automorphy
\eqref{eq:invline} but is \emph{not} $\Z^2$-periodic as a map into $\CP$, so
compactness of $\T^2$ does not by itself make it uniformly continuous
(Remark~\ref{rem:ucfails} below exhibits a continuous automorphic line field whose
lift is not uniformly continuous). Uniform continuity is instead a consequence of the
dominated dynamics. Throughout this section only the stable line $\ell_s$ is treated;
the unstable line $\ell_u$ is handled later, in Lemma~\ref{lem:inverse}, by applying
the stable-line results of this section and the next to the inverse cocycle.

\begin{lemma}[uniform continuity of the dominated line field]\label{lem:uc}
Assume strip-admissibility and a dominated splitting \eqref{eq:dom}. Then the lifted
stable line field $\ell_s\colon\R^2\to\CP$ is uniformly continuous with respect to
the Euclidean distance on $\R^2$ and the metric \eqref{eq:metric}. Quantitatively,
with $S_0:=\sup_{\Omega_{\delta_0}}\norm A$, $L_A:=S_0/\delta_0$,
$\bar C:=\max(C_0,1)$ and $\gamma_k:=(1-\kappa_0^4)\,c_0^{2kN}$ (the constants
$c_0,C_0$ of Definition~\ref{def:cocycle}),
\[
\operatorname{dist}\bigl(\ell_s(x),\ell_s(x')\bigr)\ \le\
\inf_{k\ge1}\Bigl[\,2\pi\kappa_0^{2k}+L_k\,\abs{x-x'}_1\Bigr],
\qquad
L_k:=\frac{8\pi\,kN\,\bar C^{\,2kN-1}L_A}{\gamma_k},
\]
for all $x,x'\in\R^2$ (with $\abs\cdot_1$ the $\ell^1$-distance); in particular
$\ell_s$ admits the modulus of continuity
$\omega(t):=\inf_{k\ge1}\bigl(2\pi\kappa_0^{2k}+L_kt\bigr)$, and
$\omega(t)\downarrow0$ as $t\downarrow0$.
\end{lemma}

\begin{proof}
The proof takes place entirely on the real slice; no holomorphy of any
singular-vector field is claimed or used (finite-time singular directions depend on
$A^*A$ and are in general not holomorphic). Fix $k\ge1$, write $\mathcal N:=kN$ and
$M_k(x):=A^{(\mathcal N)}(x)$ for real $x$.

\emph{Step 1: $A$ is uniformly Lipschitz on $\R^2$.} For real $x$ the closed
polydisc of polyradius $\delta_0$ centered at $x$ lies in $\Omega_{\delta_0}$, so the
Cauchy estimates give $\norm{\partial A/\partial w_i(x)}\le S_0/\delta_0=L_A$
($i=1,2$), uniformly in $x$. Integrating along a coordinate path from $x$ to $x'$
inside $\R^2$,
\[
\norm{A(x)-A(x')}\ \le\ L_A\bigl(\abs{x_1-x_1'}+\abs{x_2-x_2'}\bigr)
\ =\ L_A\abs{x-x'}_1\qquad(x,x'\in\R^2).
\]

\emph{Step 2: global singular-value bounds from the unitary sewing.} $A$ itself is
not $\Z^2$-periodic and is nowhere treated as such. But by \eqref{eq:K2} and
unitarity of $U_j$ at real points, $\norm{A(\cdot)}$ and $\co(A(\cdot))$ are
$\Z^2$-periodic on $\R^2$ (Definition~\ref{def:cocycle}), whence
$0<c_0\le\co(A)\le\norm A\le C_0<\infty$ uniformly on $\R^2$. Consequently, for
every $r\in\N$ and real $x$,
\[
c_0^{\,r}\ \le\ \sigma_2\bigl(A^{(r)}(x)\bigr)\ \le\
\sigma_1\bigl(A^{(r)}(x)\bigr)\ \le\ C_0^{\,r},
\]
since the top (resp.\ bottom) singular value is sub- (resp.\ super-) multiplicative
along the product. (Equivalently: by the block equivariance \eqref{eq:blockeq}, all
singular values of $A^{(r)}(\cdot)$ are themselves $\Z^2$-periodic on $\R^2$.)

\emph{Step 3: each block is uniformly Lipschitz, with $r$-dependent constants.}
Telescoping $A^{(r)}(x)-A^{(r)}(x')$ factor by factor, bounding the untouched
factors by $\bar C^{\,r-1}$ (Step 2) and each difference by Step 1 (the translation
$S$ preserves $\abs\cdot_1$),
\[
\norm{A^{(r)}(x)-A^{(r)}(x')}\ \le\ r\,\bar C^{\,r-1}L_A\abs{x-x'}_1 .
\]

\emph{Step 4: the least right-singular line of a block is uniformly Lipschitz.} Let
$H_k:=M_k^*M_k$, a Hermitian matrix with eigenvalues
$\sigma_1(M_k)^2\ge\sigma_2(M_k)^2$. By Lemma~\ref{lem:domcons}(2),
$\sigma_2/\sigma_1\le\kappa_0^{2k}\le\kappa_0^{2}<1$ at every real point, and
$\sigma_2\ge c_0^{\mathcal N}$ by Step 2, so
\[
\sigma_1^2-\sigma_2^2\ =\ \sigma_1^2\bigl(1-(\sigma_2/\sigma_1)^2\bigr)
\ \ge\ c_0^{2\mathcal N}\bigl(1-\kappa_0^{4k}\bigr)\ \ge\ \gamma_k\ >\ 0
\qquad\text{on }\R^2 ,
\]
a spectral gap uniform on the real slice. Hence the $\sigma_2^2$-eigenprojection
$P_k(x)$ of $H_k(x)$ is well defined and of rank one, and its range
$v_k(x)$ is the least right-singular line of $M_k(x)$. The two-regime resolvent
estimate in the proof of Lemma~\ref{lem:stripdom}(2)---which used only the spectral
calculus of Hermitian matrices with a gap, not holomorphy---gives
$\norm{P_k(x)-P_k(x')}\le(4/\gamma_k)\norm{H_k(x)-H_k(x')}$ for all real $x,x'$.
Combining with
$\norm{H_k(x)-H_k(x')}\le2\bar C^{\,\mathcal N}\norm{M_k(x)-M_k(x')}$, Step 3, and
$\operatorname{dist}(v_k(x),v_k(x'))\le\pi\norm{P_k(x)-P_k(x')}$ (rank-one
orthogonal projections, metric \eqref{eq:metric}),
\[
\operatorname{dist}\bigl(v_k(x),v_k(x')\bigr)\ \le\ L_k\abs{x-x'}_1
\qquad(x,x'\in\R^2).
\]

\emph{Step 5: $v_k\to\ell_s$ uniformly on $\R^2$.} The proof of
Lemma~\ref{lem:stripdom}(3) consumes only the domination bound
$D_{\mathcal N}(z)\le\kappa_0^{2k}$, valid at \emph{every} real $z$ by
Lemma~\ref{lem:domcons}(1); it yields
\[
\operatorname{dist}\bigl(v_k(z),\ell_s(z)\bigr)\ \le\ \pi\kappa_0^{2k}
\qquad(z\in\R^2),
\]
uniformly in $z$. The sewing does not perturb this uniformity: the bound is
pointwise and its constant is the global domination constant.

\emph{Step 6: conclusion.} For every $k\ge1$ and all real $x,x'$, the triangle
inequality through $v_k$ gives
\[
\operatorname{dist}\bigl(\ell_s(x),\ell_s(x')\bigr)\ \le\
2\pi\kappa_0^{2k}+L_k\abs{x-x'}_1 ,
\]
and taking the infimum over $k$ gives the display. Given $\varepsilon>0$, choose
first $k$ with $2\pi\kappa_0^{2k}<\varepsilon/2$ and then
$\abs{x-x'}_1<\varepsilon/(2L_k)$: this is uniform continuity---$\ell_s$ is a
uniform limit of the uniformly continuous fields $v_k$, and a uniform limit of
uniformly continuous maps is uniformly continuous.
\end{proof}

\begin{remark}[uniform continuity of the lift is not automatic]\label{rem:ucfails}
Compactness of the quotient does not make the lift of a continuous section of
$\Proj V$ uniformly continuous: the sewing is a twist, not a periodicity. Example:
$n=2$, $U_1\equiv I$, $U_2(w)=\diag\bigl(1,\eu^{2\pi\ii w_1}\bigr)$ (a valid sewing
system, \eqref{eq:K1} being immediate). In the affine coordinate $\zeta=v_2/v_1$ on
$\CP$ the automorphy of a line field reads $\zeta(w+e_1)=\zeta(w)$,
$\zeta(w+e_2)=\eu^{2\pi\ii w_1}\zeta(w)$. Choose a continuous (even smooth)
$\rho\colon[0,1]\to[0,\infty)$ with $\rho(0)=\rho(1)=0$ and $\rho(1/2)=1$, set
$\zeta(w):=\rho(w_2)$ for $0\le w_2\le1$, and extend by the automorphy:
$\zeta(w_1,w_2+n)=\eu^{2\pi\ii nw_1}\rho(w_2)$ for $w_2\in[0,1]$, $n\in\Z$. Since
the twist fixes the pole $\zeta=0$ and $\rho$ vanishes at the seams, this is a
globally continuous automorphic line field. But at height $w_2+n$ with
$\rho(w_2)=1$ the lift traverses the equator $\abs\zeta=1$ with winding rate $n$:
the points $(w_1,w_2+n)$ and $(w_1+\tfrac1{2n},w_2+n)$, at Euclidean distance
$\tfrac1{2n}\to0$, are mapped to \emph{orthogonal} lines (distance $\pi$, the
diameter). So the lift is continuous but not uniformly continuous.
Lemma~\ref{lem:uc} therefore genuinely consumes the dominated dynamics---through
the uniform convergence of the singular-direction fields---and not merely
continuity of the sewn line field plus compactness.
\end{remark}

\begin{theorem}[holomorphic extension of the invariant line]\label{thm:stripline}
Assume strip-admissibility and a dominated splitting \eqref{eq:dom}. Then there is
$\delta>0$ such that the line field $\ell_s$ of $E^s$ extends to a holomorphic map
$\tilde\ell\colon\Omega^\circ_\delta\to\CP$ with
\[
\tilde\ell(w+e_j)=U_j(w)\tilde\ell(w)\quad(j=1,2),\qquad
A^{(\mathcal N)}(w)^{-1}\tilde\ell(S^{\mathcal N}w)=\tilde\ell(w),\qquad
\tilde\ell|_{\R^2}=\ell_s ,
\]
for a suitable block length $\mathcal N$.
\end{theorem}

\begin{proof}
Fix $\varepsilon:=\varepsilon_1/4$ with $\varepsilon_1$ from
Lemma~\ref{lem:stripdom}. Choose $k$, then $\delta\le\min(\delta_1,\delta(k))$, so
that with the objects of Lemma~\ref{lem:stripdom} at block length
$\mathcal N=kN$:
(i) $\kappa\bigl(1+4/\varepsilon^2\bigr)\le\tfrac12$;
(ii) $\operatorname{dist}\bigl(v(w),E^s(\operatorname{Re}w)\bigr)\le\varepsilon_1/8$
and
$\operatorname{dist}\bigl(u(w),E^u(\operatorname{Re}(S^{\mathcal N}w))\bigr)\le
\varepsilon_1/8$ for all $w\in\Omega_\delta$ (available from
Lemma~\ref{lem:stripdom}(2)--(3): at the real point $\operatorname{Re}w$ the bounds
are $\le\pi\kappa_0^{2k}\le\varepsilon_1/16$ for large $k$, and the modulus of
continuity of $u,v$ in the $\operatorname{Im}$-direction, \emph{uniform in}
$\operatorname{Re}w$, bounds $\operatorname{dist}(v(w),v(\operatorname{Re}w))$ and
$\operatorname{dist}(u(w),u(\operatorname{Re}w))$ by $\varepsilon_1/16$ for
$\delta$ small; note $\operatorname{Re}(S^{\mathcal N}w)=
S^{\mathcal N}\operatorname{Re}w$ since $\tau$ is real---no uniform continuity of
$E^s$ or $E^u$ in the $\operatorname{Re}$-direction is needed here);
(iii) $4\kappa/\varepsilon\le\varepsilon_1/8$.
Finally put
\[
B:=\max\left\{1,\ \sup_{\Omega_{\delta_0}}
\bigl(\norm{U_1}+\norm{U_2}+\norm A\bigr)\right\}
\]
and shrink $\delta$ once more so that $C_U\delta\le1$ and
$\delta\le\varepsilon_1\delta_0/(512\pi B)$, retaining (i)--(iii).
For $r\in[\varepsilon_1/4,\,3\varepsilon_1/8]$ define the disc field
$\mathcal D_r(w):=\{\ell\in\CP:\operatorname{dist}(\ell,E^s(\operatorname{Re}w))\le r\}$,
a geodesically convex spherical disc (radius $\le3\varepsilon_1/8<\pi/2$;
$\varepsilon_1\le\pi$). Continuous sections of $\mathcal D_r$ over $\Omega_\delta$ with
the sup-distance form a complete metric space. Define the graph transform
\[
(\Phi\tilde\ell)(w):=A^{(\mathcal N)}(w)^{-1}\cdot\tilde\ell(S^{\mathcal N}w)
\qquad(\text{projective action}).
\]

\emph{Separation.} For $\ell\in\mathcal D_r(S^{\mathcal N}w)$ (any
$r\le3\varepsilon_1/8$),
\[
\operatorname{dist}(\ell,u(w))\ \ge\
\varepsilon_1-\tfrac{3\varepsilon_1}{8}-\tfrac{\varepsilon_1}{8}
=\tfrac{\varepsilon_1}{2}\ \ge\ \varepsilon
\]
by the triangle inequality, (ii), and the definition of $\varepsilon_1$; the same holds
along geodesics inside the convex disc.

\emph{Invariance and contraction.} By Lemma~\ref{lem:mobius}(1) and (iii),
$(\Phi\tilde\ell)(w)$ is within $4\kappa/\varepsilon\le\varepsilon_1/8$ of $[v(w)]$,
hence within $\varepsilon_1/4\le r$ of $E^s(\operatorname{Re}w)$ by (ii): $\Phi$ maps
sections of $\mathcal D_r$ to sections of $\mathcal D_{\varepsilon_1/4}\subseteq
\mathcal D_r$. By Lemma~\ref{lem:mobius}(2) and (i), $\Phi$ is a
$\tfrac12$-contraction in the sup-distance. Banach's fixed point theorem yields a unique
fixed section $\ell_\infty$ of $\mathcal D_{3\varepsilon_1/8}$, lying in
$\mathcal D_{\varepsilon_1/4}$.

\emph{Holomorphy.} Fix $w_0\in\Omega^\circ_\delta$. By Lemma~\ref{lem:uc} the lifted
line field $E^s=\ell_s\colon\R^2\to\CP$ is uniformly continuous; let $\omega$ be a
modulus of continuity for it and choose $\rho>0$ so small that
$\omega(2\rho)\le\min\{3\varepsilon_1/8,\ \pi/16\}$, and let
$\mathcal B\subseteq\Omega^\circ_\delta$ be the open ball of radius $\rho$ around $w_0$. Set
$c_k:=E^s\bigl(\operatorname{Re}(S^{k\mathcal N}w_0)\bigr)$ for $k\ge0$, a
\emph{constant} line for each $k$. Since $\tau$ is real,
$\operatorname{Re}(S^{k\mathcal N}w)-\operatorname{Re}(S^{k\mathcal N}w_0)
=\operatorname{Re}w-\operatorname{Re}w_0$ for every $w$ and $k$, so for all
$w\in\mathcal B$ and all $k\ge0$
\[
\operatorname{dist}\Bigl(c_k,\
E^s\bigl(\operatorname{Re}(S^{k\mathcal N}w)\bigr)\Bigr)
\ \le\ \omega\bigl(\abs{\operatorname{Re}w-\operatorname{Re}w_0}_1\bigr)
\ \le\ \omega(2\rho)\ \le\ \tfrac{3\varepsilon_1}{8},
\]
i.e.\ $c_k\in\mathcal D_{3\varepsilon_1/8}(S^{k\mathcal N}w)$ \emph{uniformly in
$k$}. This uniformity is exactly what Lemma~\ref{lem:uc} provides; it is the single
point of the proof where that lemma is consumed, and it cannot be obtained from
continuity of the sewn line field alone (Remark~\ref{rem:ucfails}). Define
$\psi_k(w):=A^{(k\mathcal N)}(w)^{-1}\cdot c_k$ on $\mathcal B$: a holomorphic
$\CP$-valued map (a fixed holomorphic invertible matrix family applied to a constant
line; $\det A^{(k\mathcal N)}\neq0$ on the $S$-invariant strip $\Omega_\delta$ by
Lemma~\ref{lem:cauchy}(3), applied to the block products). We emphasize that the
finite-time singular-direction fields $v_k$ of Lemma~\ref{lem:uc} enter only on the real slice, to produce the
modulus $\omega$; they are in general not holomorphic and are \emph{not} used as
approximants---the approximants are the maps $\psi_k$, holomorphic because the
$c_k$ are constants. Iterating the invariance and contraction pointwise along
the orbit segment---the input $c_k$ and the fixed-point value
$\ell_\infty(S^{k\mathcal N}w)$ both lie in the geodesically convex disc
$\mathcal D_{3\varepsilon_1/8}(S^{k\mathcal N}w)$, each single-block map is
$\tfrac12$-Lipschitz there (Separation and Lemma~\ref{lem:mobius}(2)), and each
single-block image lies in $\mathcal D_{\varepsilon_1/4}$ (Invariance)---gives
\[
\operatorname{dist}\bigl(\psi_k(w),\ell_\infty(w)\bigr)\le2^{-k}\pi
\qquad(w\in\mathcal B).
\]
All values of all $\psi_k$ and of $\ell_\infty$ on $\mathcal B$ lie in
$\mathcal D_{3\varepsilon_1/8}(w)$, whose center $E^s(\operatorname{Re}w)$ is within
$\omega(2\rho)\le\pi/16$ of the fixed line $E^s(\operatorname{Re}w_0)$; since
$3\varepsilon_1/8+\pi/16\le\tfrac{3\pi}8+\tfrac\pi{16}<\tfrac\pi2$ (recall
$\varepsilon_1\le\pi$), all these lines lie in a fixed spherical disc of radius
$<\pi/2$ around $E^s(\operatorname{Re}w_0)$, i.e.\ in one affine chart, with chart
coordinate and metric \eqref{eq:metric} comparable there. On that chart
$\psi_k\to\ell_\infty$ uniformly, and a uniform limit of holomorphic chart-valued
maps is holomorphic (Weierstrass). Hence $\ell_\infty$ is holomorphic on $\mathcal B$; as
$w_0$ was arbitrary, on all of $\Omega^\circ_\delta$.

\emph{Equivariance.} By Lemma~\ref{lem:cauchy}(3) the $U_j$ are invertible on the
strip; set $\ell'(w):=U_j(w)^{-1}\ell_\infty(w+e_j)$. Using \eqref{eq:blockeq},
\[
\begin{aligned}
(\Phi\ell')(w)&=A^{(\mathcal N)}(w)^{-1}U_j(S^{\mathcal N}w)^{-1}
\ell_\infty(S^{\mathcal N}w+e_j)\\
&=U_j(w)^{-1}A^{(\mathcal N)}(w+e_j)^{-1}\ell_\infty\bigl(S^{\mathcal N}(w+e_j)\bigr)
=\ell'(w),
\end{aligned}
\]
so $\ell'$ is $\Phi$-fixed. At real points $U_j$ is a spherical isometry and $E^s$ is
sewing-equivariant, so $\ell'$ is a section of $\mathcal D_{\varepsilon_1/4}$ there;
on the strip, the projective perturbation estimate and Lemma~\ref{lem:cauchy}(2)--(3)
give, uniformly for $\ell\in\CP$,
\[
\operatorname{dist}\bigl(U_j(w)^{-1}\ell,
U_j(\operatorname{Re}w)^{-1}\ell\bigr)
\le \frac{64\pi B\delta}{\delta_0}\le\frac{\varepsilon_1}{8}.
\]
Here the three operator-norm factors in the perturbation estimate are at most $2$ by
$C_U\delta\le1$ and the proof of Lemma~\ref{lem:cauchy}(3). Taking
$\ell=\ell_\infty(w+e_j)$, the real-slice isometry and sewing
equivariance put $U_j(\operatorname{Re}w)^{-1}\ell$ in
$\mathcal D_{\varepsilon_1/4}(w)$, so $\ell'$ is a section of
$\mathcal D_{3\varepsilon_1/8}$; uniqueness gives $\ell'=\ell_\infty$, i.e.\
$\ell_\infty(w+e_j)=U_j(w)\ell_\infty(w)$.

\emph{Restriction.} Over $\R^2$ the continuous section $\ell_s=E^s$ of
$\mathcal D_{\varepsilon_1/4}$ is $\Phi$-fixed ($\mathcal N$-fold invariance), and the
contraction argument applies verbatim over the $S^{\mathcal N}$-invariant subset
$\R^2\subseteq\Omega_\delta$; uniqueness there forces
$\ell_\infty|_{\R^2}=\ell_s$. Finally $\tilde\ell:=\ell_\infty$ satisfies
$A^{(\mathcal N)}(w)^{-1}\tilde\ell(S^{\mathcal N}w)=\tilde\ell(w)$ by fixedness.
\end{proof}

\begin{remark}\label{rem:onestep}
On $\R^2$, one-step invariance $A(z)\ell_s(z)=\ell_s(Sz)$ holds by hypothesis. We do
not need (and do not claim) one-step invariance of $\tilde\ell$ off the real slice; the
$\mathcal N$-block fixedness and the sewing equivariance are what the sequel consumes.
\end{remark}

\section{The Stein normalization: an analytic charge gauge}\label{sec:stein}

\begin{lemma}[the strip torus is Stein; its line bundles]\label{lem:stein}
Let $X_\delta:=\Omega^\circ_\delta/\Z^2$ (the $\Z^2$-action $w\mapsto w+e_j$ is free,
properly discontinuous, holomorphic). Then:
\begin{enumerate}
\item $X_\delta$ is biholomorphic to a product of two plane annuli via
$(w_1,w_2)\mapsto(\eu^{2\pi\ii w_1},\eu^{2\pi\ii w_2})$; in particular $X_\delta$ is
Stein;
\item the inclusion $\T^2\hookrightarrow X_\delta$ (real slice) is a deformation
retract; restriction gives an isomorphism $H^2(X_\delta;\Z)\cong H^2(\T^2;\Z)\cong\Z$,
and topological complex line bundles on either space are classified by $c_1$;
\item the exponential sequence and Cartan's Theorem~B give
$\Pic(X_\delta)\cong H^2(X_\delta;\Z)$ via $c_1$; in particular a holomorphic line
bundle on $X_\delta$ is holomorphically trivial iff it is topologically trivial;
\item likewise $\Pic(\Omega^\circ_\delta)=0$: every holomorphic line bundle on the
contractible Stein tube $\Omega^\circ_\delta$ is holomorphically trivial.
\end{enumerate}
\end{lemma}

\begin{proof}
(1) $\zeta\mapsto\eu^{2\pi\ii\zeta}$ maps $\{\abs{\operatorname{Im}\zeta}<\delta\}/\Z$
bijectively holomorphically onto the annulus
$\{\eu^{-2\pi\delta}<\abs q<\eu^{2\pi\delta}\}\subset\C$. A product of planar domains
is a domain of holomorphy, hence Stein (see \cite{Hormander,GrauertRemmert}); so is
the tube $\Omega^\circ_\delta$ (convex). Locally, branches of
$(2\pi\ii)^{-1}\log$ give the inverse of $\zeta\mapsto\eu^{2\pi\ii\zeta}$, and any two
such branches differ by an integer.
(2) Scale the imaginary parts to zero: $(w,t)\mapsto\operatorname{Re}w+ (1-t)\ii
\operatorname{Im}w$ is a retraction homotopy; homotopy invariance of cohomology and the
classification of line bundles by $c_1$ \cite{HatcherVB,Husemoller} give the claim.
(3) From $0\to\Z\to\mathcal O\to\mathcal O^\times\to0$ and
$H^1(X,\mathcal O)=H^2(X,\mathcal O)=0$ on a Stein manifold (Cartan's Theorem~B;
\cite[Ch.~VII]{Hormander}, \cite{GrauertRemmert}), the connecting map
$c_1\colon H^1(X,\mathcal O^\times)\to H^2(X;\Z)$ is an isomorphism.
(4) Combine (3) for $\Omega^\circ_\delta$ with $H^2(\Omega^\circ_\delta;\Z)=0$.
\end{proof}

\begin{theorem}[real-analytic normalized gauge]\label{thm:gauge}
Assume strip-admissibility and a dominated splitting, and let
$\tilde\ell$ be the holomorphic equivariant extension of $\ell_s$ from
Theorem~\ref{thm:stripline} on $\Omega^\circ_\delta$. Then there is a real-analytic
unit section $e_{\mathrm{an}}\colon\R^2\to\C^2$ of $\ell_s$ satisfying the normalized
automorphy \eqref{eq:K3} with the same charge $m=m(\ell_s)$ as in
Lemma~\ref{lem:charge}.
\end{theorem}

\begin{proof}
\emph{A holomorphic frame with holomorphic factors.} The pullback
$\tilde\ell^{\,*}\gamma$ of the tautological bundle under
$\tilde\ell\colon\Omega^\circ_\delta\to\CP$ is holomorphically trivial
(Lemma~\ref{lem:stein}(4)), so there is a holomorphic nonvanishing
$s_0\colon\Omega^\circ_\delta\to\C^2$ with $s_0(w)\in\tilde\ell(w)$. Since
$\tilde\ell(w+e_j)=U_j(w)\tilde\ell(w)$ and $U_j(w)$ is invertible on the strip,
$s_0(w+e_j)=t_j(w)U_j(w)s_0(w)$ defines holomorphic zero-free
$t_j\colon\Omega^\circ_\delta\to\C^\times$. Computing $s_0(w+e_1+e_2)$ along the two
edge orders and using \eqref{eq:K1} (valid on the strip by
Lemma~\ref{lem:cauchy}(1)) gives the consistency
$t_1(w+e_2)t_2(w)=t_2(w+e_1)t_1(w)$.

\emph{The class of the factors, computed on the real slice.} On $\R^2$ let
$\hat e:=s_0|_{\R^2}/\abs{s_0}$; it is a continuous unit section of $\ell_s$, and its
automorphy cocycle is $\hat\sigma_j=t_j/\abs{t_j}$ with
$\abs{t_j}=\abs{s_0(\cdot+e_j)}/\abs{s_0}$ (unitarity of $U_j$ at real points). Let
$e_{\mathrm{ref}}$ be the continuous unit section from Lemma~\ref{lem:charge}. Since
$\hat e=\mu e_{\mathrm{ref}}$ for a continuous unimodular $\mu$, $\hat\sigma$ is
cohomologous to the cocycle of that lemma, so
by Lemma~\ref{lem:norm} it normalizes with the \emph{same} integer $m$:
$\hat\sigma_j=\rho^{(m)}_j\,\partial_j\omega$ for some continuous
$\omega\colon\R^2\to S^1$. Hence, on $\R^2$,
\[
\frac{t_j}{\rho^{(m)}_j}=\partial_j\omega\cdot
\frac{\abs{s_0(\cdot+e_j)}}{\abs{s_0}}
=\partial_j\Bigl(\omega\,\abs{s_0}\Bigr),
\]
a continuous $\C^\times$-coboundary.

\emph{The factor bundle and its trivialization.} The pair
$c_j:=t_j/\rho^{(m)}_j$ (with $\rho^{(m)}_2(w)=\eu^{2\pi\ii m w_1}$ entire and
zero-free) is a consistent holomorphic $\C^\times$-factor system on
$\Omega^\circ_\delta$; let $\mathcal N_m$ be the holomorphic line bundle on
$X_\delta$ obtained as the quotient of $\Omega^\circ_\delta\times\C$ by the
$\Z^2$-action $(w,\xi)\mapsto(w+e_j,c_j(w)\xi)$ (consistency is exactly the cocycle
condition; the action is free and properly discontinuous, and the quotient is locally
trivial over evenly covered charts). By construction, holomorphic sections of
$\mathcal N_m$ over $X_\delta$ are precisely holomorphic $h$ on
$\Omega^\circ_\delta$ with $h(w+e_j)=c_j(w)h(w)$, and continuous sections of
$\mathcal N_m|_{\T^2}$ are continuous functions on $\R^2$ with the same automorphy.
The displayed coboundary $\omega\abs{s_0}$ is such a function, and it is zero-free
because $\abs\omega=1$ and $s_0$ is nonvanishing; a line bundle with a zero-free
continuous section is topologically trivial, so $c_1(\mathcal N_m|_{\T^2})=0$; by Lemma~\ref{lem:stein}(2) and
naturality of $c_1$, $\mathcal N_m$ is topologically trivial on $X_\delta$; by
Lemma~\ref{lem:stein}(3) it is \emph{holomorphically} trivial. Thus there is a
holomorphic zero-free $h\colon\Omega^\circ_\delta\to\C^\times$ with
$h(w+e_j)=c_j(w)h(w)$.

\emph{Normalization.} Set $s:=h^{-1}s_0$: holomorphic, zero-free, $s(w)\in
\tilde\ell(w)$, with
\begin{equation}\label{eq:analyticautomorphy}
s(w+e_j)=\rho^{(m)}_j(w)\,U_j(w)\,s(w)\qquad(j=1,2).
\end{equation}
At real points $\rho^{(m)}_j$ is unimodular and $U_j$ unitary, so $\abs s$ is
$\Z^2$-periodic; it is real-analytic and positive, being $\sqrt{\ip ss}$ for
holomorphic $s$. Then $e_{\mathrm{an}}:=s|_{\R^2}/\abs s$ is a real-analytic unit
section of $\ell_s$ satisfying \eqref{eq:K3} with the integer $m$.
\end{proof}

\begin{remark}
The Oka--Grauert principle is not needed: both trivializations above follow from
Cartan's Theorem~B via the exponential sequence. Degree plays no role in the argument
---the construction works for every value of $m$; this is exactly what makes the
normalization useful for nontrivial $L$.
\end{remark}

\section{Diophantine cohomology and the proof of Theorem B}\label{sec:analytic}

\begin{definition}[Diophantine condition]\label{def:DC}
$\tau=(\alpha,\beta)$ satisfies $\mathrm{DC}(c,\nu)$, $c>0$, $\nu\ge1$, if
\[
\norm{\nu_1'\alpha+\nu_2'\beta}_{\R/\Z}\ \ge\ \frac{c}{(\abs{\nu_1'}+\abs{\nu_2'})^{\nu}}
\qquad\text{for all }(\nu_1',\nu_2')\in\Z^2\setminus\{0\}.
\]
Here $\norm{x}_{\R/\Z}:=\min_{n\in\Z}\abs{x-n}$; for
$\nu'=(\nu_1',\nu_2')$, we use $\abs{\nu'}_1:=\abs{\nu_1'}+\abs{\nu_2'}$ and
$\abs{\nu'}_\infty:=\max\{\abs{\nu_1'},\abs{\nu_2'}\}$.
$\mathrm{DC}(c,\nu)$ implies rational independence of $1,\alpha,\beta$ (a rational
relation would make some $\norm{\nu'\cdot\tau}_{\R/\Z}=0$), hence ergodicity of $S$.
\end{definition}

\begin{lemma}[a cubic example]\label{lem:cubic}
Let $\theta=2^{1/3}$ and $\tau^*=(\theta-1,\ \theta^2-1)$. Then
$\tau^*\in\mathrm{DC}(1/14,\,2)$.
\end{lemma}

\begin{proof}
For $(\nu_1',\nu_2')\neq0$, $\norm{\nu_1'(\theta-1)+\nu_2'(\theta^2-1)}_{\R/\Z}
=\norm{\nu_1'\theta+\nu_2'\theta^2}_{\R/\Z}$. Let $p$ be a nearest integer to
$\nu_1'\theta+\nu_2'\theta^2$ and $\gamma:=\nu_1'\theta+\nu_2'\theta^2-p\in\Z[\theta]$,
an algebraic integer of the field $\Q(\theta)$; $\gamma\neq0$ since $1,\theta,\theta^2$
are linearly independent over $\Q$ (irreducibility of $x^3-2$). Hence
$\abs{N(\gamma)}=\abs{\gamma\gamma'\gamma''}\ge1$, where the conjugates are obtained by
$\theta\mapsto\theta\omega,\theta\omega^2$ ($\omega=\eu^{2\pi\ii/3}$). If
$\abs\gamma\ge\tfrac12$ the claimed bound is trivial. If $\abs\gamma\le\tfrac12$, then
$\abs p\le\tfrac12+\theta\abs{\nu_1'}+\theta^2\abs{\nu_2'}$ and so, for each conjugate,
\[
\abs{\gamma^{(i)}}\le\abs p+\theta\abs{\nu_1'}+\theta^2\abs{\nu_2'}
\le\tfrac12+2\theta^2\bigl(\abs{\nu_1'}+\abs{\nu_2'}\bigr)
\le\bigl(\tfrac12+2\theta^2\bigr)\bigl(\abs{\nu_1'}+\abs{\nu_2'}\bigr),
\]
whence
$\abs\gamma\ge\bigl(\tfrac12+2\theta^2\bigr)^{-2}
(\abs{\nu_1'}+\abs{\nu_2'})^{-2}\ge\tfrac1{14}(\abs{\nu_1'}+\abs{\nu_2'})^{-2}$,
using $\tfrac12+2\theta^2<3.68$. Such badly approximable linear forms from cubic
fields are classical; see \cite{CSD1955,Davenport1964,Schmidt1980}.
\end{proof}

\begin{lemma}[analytic cohomological equation]\label{lem:solve}
Let $\tau\in\mathrm{DC}(c,\nu)$ and let $h\colon\T^2\to\R$ be real analytic. Then there
is a real-analytic $v\colon\T^2\to\R$ with
\[
v(w)-v(Sw)=h(w)-\langle h\rangle\qquad(w\in\T^2).
\]
\end{lemma}

\begin{proof}
A real-analytic $\Z^2$-periodic function extends holomorphically to some
$\Omega^\circ_{\delta'}$ \cite{KrantzParks}, and shifting the contour of integration in
each variable (periodicity cancels the boundary terms) gives the exponential decay
$\abs{\hat h_{\nu'}}\le C\eu^{-2\pi\delta''\abs{\nu'}_\infty}$ with
$\delta''=\delta'/2$. Comparing Fourier coefficients, the equation forces
\[
\hat v_{\nu'}=\frac{\hat h_{\nu'}}{1-\eu^{2\pi\ii\nu'\cdot\tau}}\quad(\nu'\neq0),
\qquad \hat v_0:=0 .
\]
Small divisors: $\abs{1-\eu^{2\pi\ii\nu'\cdot\tau}}=2\abs{\sin\pi\nu'\cdot\tau}\ge
4\norm{\nu'\cdot\tau}_{\R/\Z}\ge4c\,(\abs{\nu_1'}+\abs{\nu_2'})^{-\nu}$ (using
$\abs{\sin\pi x}\ge2\norm x_{\R/\Z}$). Hence
$\abs{\hat v_{\nu'}}\le\tfrac{C}{4c}(\abs{\nu_1'}+\abs{\nu_2'})^{\nu}
\eu^{-2\pi\delta''\abs{\nu'}_\infty}$, still exponentially decaying. On every strictly
smaller strip the polynomial factor is absorbed into a smaller exponential; hence the
Fourier series $v:=\sum\hat v_{\nu'}\eu^{2\pi\ii\nu'\cdot w}$ converges normally there
and defines a holomorphic extension, so its real restriction is real analytic. The
function $v$ is real-valued because $\hat v_{-\nu'}=\overline{\hat v_{\nu'}}$ (from
$\hat h_{-\nu'}=\overline{\hat h_{\nu'}}$ and conjugation of the divisor). Finally the
defining relation holds pointwise by continuity.
\end{proof}

\begin{lemma}[the $E^u$ case by inversion]\label{lem:inverse}
Assume strip-admissibility and a dominated splitting for $(U,A)$ over $S$. Define
$\hat A(w):=A(S^{-1}w)^{-1}$ over $\hat S:=S^{-1}$. Then $(U,\hat A)$ is
strip-admissible over $\hat S$ (same sewing), the pair
$(\hat E^s,\hat E^u):=(E^u,E^s)$ is a dominated splitting for $(U,\hat A)$ with the
same constants. If $\tau\in\mathrm{DC}(c,\nu)$, then $-\tau\in\mathrm{DC}(c,\nu)$.
Consequently
Lemma~\ref{lem:uc}, Theorem~\ref{thm:stripline} and Theorem~\ref{thm:gauge}---all
already proved, for the stable line of the data to which they are applied---apply to
$(U,\hat A)$, whose stable line is $\ell_u$; they produce a holomorphic equivariant
extension of $\ell_u$ on a strip and a real-analytic
normalized unit section $e_{u,\mathrm{an}}$ of $\ell_u$ with charge $m(\ell_u)$.
\end{lemma}

\begin{proof}
Equivariance: from \eqref{eq:K2} at $z=S^{-1}w$,
$A(z+e_j)U_j(z)=U_j(w)A(z)$; solving for the inverse,
$A(z+e_j)^{-1}U_j(w)=U_j(z)A(z)^{-1}$, i.e.\
$\hat A(w+e_j)U_j(w)=U_j(\hat Sw)\hat A(w)$. Invariance:
$A(S^{-1}w)E^u(S^{-1}w)=E^u(w)$ gives $\hat A(w)E^u(w)=E^u(\hat Sw)$; likewise for
$E^s$. Strip bounds: choose the positive smaller half-width
$\delta_1\in(0,\delta_0)$ supplied by Lemma~\ref{lem:cauchy}(3), so that
$\inf_{\Omega_{\delta_1}}\abs{\det A}>0$. On $\Omega_{\delta_1}$ the shifted inverse
$\hat A=\adj A(S^{-1}\cdot)/\det A(S^{-1}\cdot)$ is holomorphic and bounded because
$\adj A$ is bounded there. Domination: with
$M:=A^{(N)}(S^{-N}w)$ one has $\hat A^{(N)}(w)=M^{-1}$, and since norms of inverses of
maps between invariant lines are reciprocals,
\[
\frac{\norm{\hat A^{(N)}(w)|_{E^u(w)}}}{\co\bigl(\hat A^{(N)}(w)|_{E^s(w)}\bigr)}
=\frac{1/\co\bigl(M|_{E^u(S^{-N}w)}\bigr)}{1/\norm{M|_{E^s(S^{-N}w)}}}
=D_N(S^{-N}w)\le\kappa_0^2 .
\]
If $\tau\in\mathrm{DC}(c,\nu)$, then
$\norm{\nu'\cdot(-\tau)}_{\R/\Z}=\norm{\nu'\cdot\tau}_{\R/\Z}$ for every
$\nu'\in\Z^2\setminus\{0\}$, so $-\tau\in\mathrm{DC}(c,\nu)$. Finally, the verification
just given consumes only \eqref{eq:K2}, Definition~\ref{def:dom} and
Lemma~\ref{lem:cauchy}; Lemma~\ref{lem:uc}, Theorem~\ref{thm:stripline} and
Theorem~\ref{thm:gauge} enter only as previously proved statements now applied to the
inverse data, so the dependency order
(\ref{lem:uc}$\,\to\,$\ref{thm:stripline}$\,\to\,$\ref{thm:gauge}$\,\to\,$%
present lemma) is linear and no circularity arises.
\end{proof}

\begin{lemma}[analytic multiplier data on either line]\label{lem:analytic}
Let $(U_1,U_2,A)$ be strip-admissible with values in $U(2)$, $GL(2,\C)$
\textup{(Definition~\ref{def:stripadm})}, admitting a dominated splitting
$E^s\oplus E^u$ \textup{(Definition~\ref{def:dom})}, over $\tau\in\mathrm{DC}(c,\nu)$
\textup{(Definition~\ref{def:DC})}, and let $L\in\{E^s,E^u\}$ be either line, of
\emph{any} charge $m(L)$. Then the line field of $L$ admits a real-analytic normalized
unit section $e_{\mathrm{an}}$, with charge $m(L)$; in that gauge the multiplier $q$
is real analytic and zero free, the data $\log\abs q$ and $g$ are real analytic on
$\T^2$, and the cohomological equations of Lemma~\ref{lem:solve} for $h=\log\abs q$ and
for $h=g$ both have real-analytic solutions. In particular \textup{(H2)} holds.
\end{lemma}

\begin{proof}
By Theorem~\ref{thm:gauge} (for $L=E^s$) or Lemma~\ref{lem:inverse} (for $L=E^u$;
note that only the \emph{gauge} is produced by the inverse dynamics---the multiplier
below is taken with the forward cocycle $A$), the line field $\ell$ of $L$ has a
real-analytic normalized unit section $e_{\mathrm{an}}$ with charge $m=m(\ell)$, and
$c_1(L)[\T^2]=m$ by Lemma~\ref{lem:degree}.
With $e:=e_{\mathrm{an}}$ in Lemma~\ref{lem:mult}, the
multiplier $q(w)=\ip{e(Sw)}{A(w)e(w)}$ is real analytic (products, compositions and
conjugates of real-analytic maps; $\abs{s}=\sqrt{\ip ss}$ is real analytic and positive
for the holomorphic $s$ of Theorem~\ref{thm:gauge}) and zero-free, with the exact
quasi-periodicity \eqref{eq:qquasi} (that derivation consumes only \eqref{eq:K3} and
\eqref{eq:K2}). Hence $p=(q/\abs q)\eu^{2\pi\ii m\alpha w_2}$ is real analytic,
$\Z^2$-periodic and unimodular; its zero-winding phase $g$ (Lemma~\ref{lem:mult}) is
continuous and locally of the form
$\tfrac1{2\pi\ii}\log\bigl(p\,\eu^{-2\pi\ii k\cdot w}\bigr)$ for a holomorphic branch of
$\log$ composed with a real-analytic zero-free function, hence real analytic. Lemma
\ref{lem:solve} now applies to each of the real-analytic real-valued functions
$\log\abs q$ and $g$, and its solution for $h=g$ is in particular a measurable solution
of \eqref{eq:H2}, so \textup{(H2)} holds. No step used any hypothesis on $c_1(L)$.
Theorem~\ref{thm:stripline} and Lemma~\ref{lem:solve} are charge-independent, while
Theorem~\ref{thm:gauge} and Lemma~\ref{lem:inverse} impose no restriction on the charge
and construct the required gauge for every charge. See also Remark~7.3.
\end{proof}

\begin{theorem}[analytic--Diophantine theorem]\label{thm:B}
Let $(U_1,U_2,A)$ be strip-admissible with values in $U(2)$, $GL(2,\C)$
\textup{(Definition~\ref{def:stripadm})}, admitting a dominated splitting
$E^s\oplus E^u$ \textup{(Definition~\ref{def:dom})}, over $\tau\in\mathrm{DC}(c,\nu)$
\textup{(Definition~\ref{def:DC})}. Let $L\in\{E^s,E^u\}$ with $c_1(L)\neq0$. Then $L$
carries no nonzero measurable eigensection: for every $\lambda\in\C$, every measurable
section $F$ carried by the line field of $L$ with $A(w)F(w)=\lambda F(w+\tau)$ a.e.\
vanishes a.e.
\end{theorem}

\begin{proof}
By Lemma~\ref{lem:analytic}, hypothesis \textup{(H2)} holds for $(A,\ell)$. All
hypotheses of Theorem~\ref{thm:A} are then verified ($\mathrm{DC}$ gives ergodicity), and
the exclusion follows.
\end{proof}

The same lemma discharges \textup{(H2)} on a charge-zero line, where
Theorem~\ref{thm:B} says nothing; combined with Proposition~\ref{prop:winding} this
disposes of the case $m=0$, $k\neq0$, and with Theorem~\ref{thm:Aex} it settles the
remaining case.

\begin{theorem}[analytic--Diophantine existence]\label{thm:Bex}
Let $(U_1,U_2,A)$ be strip-admissible with values in $U(2)$, $GL(2,\C)$
\textup{(Definition~\ref{def:stripadm})}, admitting a dominated splitting
$E^s\oplus E^u$ \textup{(Definition~\ref{def:dom})}, over
$\tau\in\mathrm{DC}(c,\nu)$ \textup{(Definition~\ref{def:DC})}. Let
$L\in\{E^s,E^u\}$ with
\[
m(L)=0\quad\text{(equivalently }c_1(L)=0\text{)}\qquad\text{and}\qquad k(L)=0 .
\]
Then $L$ carries a nonzero measurable eigensection. Explicitly, in the real-analytic
gauge $e_{\mathrm{an}}$ of Lemma~\ref{lem:analytic}, with $v_1,v_2$ the real-analytic
solutions of Lemma~\ref{lem:solve} for $h=\log\abs q$ and $h=g$,
\[
F_0=\eu^{-(v_1+2\pi\ii v_2)}\,e_{\mathrm{an}},\qquad
\lambda_0=\exp\Bigl(\int_{\T^2}\log\abs q\,d\mu\Bigr)\,
\exp\Bigl(2\pi\ii\int_{\T^2}g\,d\mu\Bigr),
\]
and $F_0$ is real analytic, nowhere vanishing, lies in $\mathcal H_U$, and
$\abs{\lambda_0}=\eu^{\lambda_L}$, the exponential of the Lyapunov exponent of $L$
\textup{(Lemma~\ref{lem:domcons}(3))}.
\end{theorem}

\begin{proof}
Lemma~\ref{lem:analytic} supplies $e_{\mathrm{an}}$ and the real analyticity of $q$,
$\log\abs q$ and $g$; since $m=0$, \eqref{eq:K3} reads
$e_{\mathrm{an}}(w+e_j)=U_j(w)e_{\mathrm{an}}(w)$ and $q$ is a real-analytic zero-free
function on $\T^2$ (Lemma~\ref{lem:log}). Since also $k=0$, Lemma~\ref{lem:log} gives the
global real-analytic branch $h=\log\abs q+2\pi\ii g$ with $\eu^h=q$. Put
$\Psi:=-(v_1+2\pi\ii v_2)$, real analytic; by Lemma~\ref{lem:solve},
\[
\Psi(Sw)-\Psi(w)
=\bigl(\log\abs q-\langle\log\abs q\rangle\bigr)+2\pi\ii\bigl(g-\langle g\rangle\bigr)
=h(w)-\langle h\rangle ,
\]
so \textup{(E2)} holds with a bounded $\Psi$ and Theorem~\ref{thm:Aex} applies, giving
$F_0=\eu^{\Psi}e_{\mathrm{an}}$ real analytic and nowhere vanishing with eigenvalue
$\lambda_0=\eu^{\langle h\rangle}$. Finally
$\abs{\lambda_0}=\eu^{\langle\log\abs q\rangle}=\eu^{\lambda_L}$ by
Lemma~\ref{lem:domcons}(3).
\end{proof}

\begin{remark}[what was used where]\label{rem:whereused}
Domination and analyticity enter only through
Theorems~\ref{thm:stripline}--\ref{thm:gauge} (construction and analytic
normalization of the line and its phase data) and Lemma~\ref{lem:solve} (the
small-divisor solve);
the obstruction itself is Theorem~\ref{thm:A}. The Diophantine condition enters only
in Lemma~\ref{lem:solve}. In particular any other route to (H2) (for a specific
cocycle, e.g.\ by inspection) yields the same conclusion without domination; see
Example~\ref{ex:both}.
\end{remark}

\section{Carried lines, examples, and limitations}\label{sec:corollaries}

\begin{lemma}[classification of carrying lines]\label{lem:carried}
Let $(U_1,U_2,A)$ be continuous data as in Section~\ref{sec:setting} admitting a
dominated splitting $E^s\oplus E^u$ \textup{(Definition~\ref{def:dom})}, with $S$
ergodic (no analyticity or Diophantine hypothesis). Let $F$ be a measurable
eigensection, $F\neq0$ on a positive-measure set, with eigenvalue $\lambda$. Then
$\lambda \neq 0$, $F\neq0$ a.e., and either $F(w)\in E^s(w)$ for a.e.\ $w$ or
$F(w)\in E^u(w)$ for a.e.\ $w$.
\end{lemma}

\begin{proof}
As in Lemma~\ref{lem:scalar}(C3), $\{F=0\}$ is $\Z^2$- and $S$-invariant modulo null
(sewing; invertibility of $A$ with $\lambda\neq0$---and $\lambda=0$ would force
$F=0$ a.e.), hence null. Let $e_s,e_u$ be continuous normalized unit sections of the
two line fields (Lemma~\ref{lem:charge}) with multipliers $q_s,q_u$
(Lemma~\ref{lem:mult}), and write $F=c_se_s+c_ue_u$ with measurable coefficients
$c_s,c_u$. Inserting into \eqref{eq:eigen2} and comparing components in the frame at
$Sw$,
\[
c_s(Sw)=\lambda^{-1}q_s(w)c_s(w),\qquad c_u(Sw)=\lambda^{-1}q_u(w)c_u(w)
\qquad\text{a.e.}
\]
So $\{c_s=0\}$ and $\{c_u=0\}$ are $S$-invariant modulo null; moreover the sewing
gives $c_i(w+e_j)=\sigma^{(i)}_j(w)^{-1}c_i(w)$ a.e.\ with unimodular
$\sigma^{(i)}_j$, so both sets are $\Z^2$-invariant modulo null and descend to
$\T^2$; by ergodicity each has measure $0$ or $1$. If
$\mu\{c_u=0\}=1$, $F$ is carried by $E^s$; if $\mu\{c_s=0\}=1$, by $E^u$. Otherwise
$c_s\neq0\neq c_u$ a.e.; put $m^\sharp:=c_u/c_s$, measurable with
$\abs{m^\sharp}\in(0,\infty)$ a.e.\ and
$\abs{m^\sharp}$ descending to $\T^2$ (the automorphy factors of $c_s,c_u$ are
unimodular). Then
$\log\abs{m^\sharp(S^nw)}=\log\abs{m^\sharp(w)}+\sum_{j<n}\log\abs{q_u/q_s}(S^jw)$
for a.e.\ $w$,
and by Lemma~\ref{lem:domcons}(3) and the pointwise ergodic theorem
\cite{Walters1982} the sums tend to $+\infty$ a.e.; but by Poincar\'e recurrence
\cite{Walters1982} a.e.\ point of a positive-measure set
$\{\abs{\log\abs{m^\sharp}}\le M\}$ returns to it infinitely often---a contradiction.
\end{proof}

\begin{corollary}[operator-level statement]\label{cor:C}\leavevmode\par
Assume the strip-admissible rank-two data, dominated splitting, and Diophantine translation
of Theorem~\ref{thm:B}. Then:
\begin{enumerate}
\item every measurable eigensection of $T_A$ is carried a.e.\ by $E^s$ or by $E^u$; this
clause uses only the hypotheses of Lemma~\ref{lem:carried}, not the analytic, Diophantine,
or Chern/charge assumptions of Theorem~\ref{thm:B};
\item if $c_1(E^s)\neq0$ and $c_1(E^u)\neq0$, then $T_A$ has no measurable
eigensections at all; in particular the point spectrum of $T_A$ on $\mathcal H_U$ is
empty;
\item if exactly one of the classes vanishes, every measurable eigensection is carried
a.e.\ by the topologically trivial line.
\end{enumerate}
\end{corollary}

\begin{proof}
Combine Lemma~\ref{lem:carried} with Theorem~\ref{thm:B} applied to each line of
nonzero class.
\end{proof}

\begin{remark}[Whitney budget]\label{rem:whitney}
Write
\[
m_i:=m(E^i)=c_1(E^i)[\T^2],\qquad i\in\{s,u\},
\]
and choose normalized unit sections $e_s,e_u$ as in Lemma~\ref{lem:charge}. Since
$E^s\oplus E^u=V$, Lemma~\ref{lem:degree} and the Whitney sum formula give
$m_s+m_u=c_1(V)[\T^2]$; equivalently, this can be read off from
$G:=\det[e_s\,e_u]$, whose automorphy multiplies the two normalized cocycles. (In the
same gauge there is an analogous determinant-winding identity, the two-frequency version
of a familiar one-frequency identity; cf.\ \cite[Lemma~6.4]{AJS2014}. That
determinant-winding identity is not used later, whereas the Whitney identity is used in
Corollary~\ref{cor:D}.) Thus on a bundle with
$c_1(V)[\T^2]=1$, case (2) of Corollary~\ref{cor:C} occurs exactly when
$(m_s,m_u)\notin\{(0,1),(1,0)\}$, and case (3) otherwise.
\end{remark}

\subsection*{The dichotomy}

Theorem~\ref{thm:B}, Proposition~\ref{prop:winding} and Theorem~\ref{thm:Bex} together
determine exactly which lines carry eigensections. We first record that there are no lines
other than the two.

\begin{lemma}[the splitting exhausts the invariant lines]\label{lem:onlytwo}
Let $n=2$ and let $(U,A)$ be continuous data admitting a dominated splitting
\textup{(Definition~\ref{def:dom})} with $S$ ergodic. Then every continuous invariant line
equals $\ell_s$ or $\ell_u$.
\end{lemma}

\begin{proof}
Let $\ell$ be a continuous invariant line and put
$d_s(w):=\operatorname{dist}(\ell(w),\ell_s(w))$,
$d_u(w):=\operatorname{dist}(\ell(w),\ell_u(w))$ in the metric \eqref{eq:metric}. All three
line fields transform under $w\mapsto w+e_j$ by the \emph{same} map $U_j(w)$, unitary at
real points and hence a spherical isometry, so $d_s,d_u$ are $\Z^2$-invariant and descend
to continuous functions on $\T^2$. (This reduction is not cosmetic: by
Remark~\ref{rem:ucfails} one may not argue on $\R^2$ by compactness.)

If $Z:=\{d_s=0\}\neq\emptyset$, then $Z$ is closed, and $S$-invariant because
$\ell(w)=\ell_s(w)$ forces $\ell(Sw)=A(w)\ell(w)=\ell_s(Sw)$. An ergodic translation of
$\T^2$ is minimal: $H:=\overline{\{n\tau:n\in\Z\}}$ is a closed subgroup. If $H$ were
proper, a nonzero character would annihilate it, giving $\nu\cdot\tau\in\Z$. Hence
$H=\T^2$ and every orbit is dense, so $Z=\T^2$ and $\ell=\ell_s$.

Otherwise $d_s\ge\varepsilon>0$ on $\T^2$ by compactness. Write
$\ell=\C\bigl(a\,e_s+b\,e_u\bigr)$ in the frame of normalized unit sections of the two
lines; then $b\neq0$ everywhere, and on $\{d_s\ge\varepsilon\}$ the chart $[a:b]$ is
uniformly comparable to the metric \eqref{eq:metric}, because the frame is uniformly
transversal: $x\mapsto\operatorname{dist}(E^s(x),E^u(x))$ is positive and continuous,
and unitary sewing makes it $\Z^2$-periodic, so it descends to the compact torus and has
a strictly positive minimum $\varepsilon_1$. In particular $\abs{a/b}\le C$ uniformly. With
$\mathcal N=kN$ and $q_s,q_u$ the multipliers of the two lines, invariance gives
\[
\ell(S^{\mathcal N}w)=\C\Bigl(a\,q_s^{(\mathcal N)}(w)\,e_s(S^{\mathcal N}w)
+b\,q_u^{(\mathcal N)}(w)\,e_u(S^{\mathcal N}w)\Bigr),\qquad
\Bigl\lvert\frac{q_s^{(\mathcal N)}}{q_u^{(\mathcal N)}}\Bigr\rvert=D_{\mathcal N}(w)
\le\kappa_0^{2k}
\]
at \emph{every} real $w$, by Lemma~\ref{lem:domcons}(1) and the multiplicativity of norms
and co-norms of maps between lines. Converting back through the same uniform comparison,
$\sup_{\T^2}d_u\circ S^{\mathcal N}\le C'\kappa_0^{2k}$; but $S^{\mathcal N}$ is a
bijection of $\T^2$ and $d_u$ descends, so $\sup_{\T^2}d_u\le C'\kappa_0^{2k}$ for every
$k$, whence $d_u\equiv0$ and $\ell=\ell_u$.
\end{proof}

\begin{theorem}[character twists and classification on a charge-free, winding-free line]
\label{thm:class}
Let $(U_1,U_2,A)$ be continuous data as in Section~\ref{sec:setting}, and let $L$ be a
continuous invariant line, of arbitrary charge, carrying a nonzero measurable eigensection
$F_0$ with eigenvalue $\lambda_0$. Then:
\begin{enumerate}
\item for every $\nu\in\Z^2$, $F_\nu:=\eu^{2\pi\ii\nu\cdot w}F_0$ is an eigensection
carried by $L$ with eigenvalue $\lambda_\nu:=\lambda_0\eu^{-2\pi\ii\nu\cdot\tau}$;
\item if, in addition, $m(L)=0$ and the hypotheses of Theorem~\ref{thm:Bex} hold, with
$F_0,\lambda_0$ as constructed there (in particular, $k(L)=0$), then conversely every
nonzero measurable eigensection $F$ carried by the line field of $L$, with eigenvalue
$\lambda$, satisfies $\lambda=\lambda_\nu$ and $F=cF_\nu$ a.e.\ for a unique $\nu\in\Z^2$
and some constant $c\neq0$.
\end{enumerate}
Under the hypotheses of part~\textup{(2)}, the set of eigenvalues realised on $L$ is the coset
$\Lambda_L=\lambda_0\{\eu^{-2\pi\ii\nu\cdot\tau}:\nu\in\Z^2\}$, a dense subset of the
circle $\abs\lambda=\eu^{\lambda_L}$; each is geometrically simple within $L$, i.e.\ the
space of measurable eigensections carried by $L$ at that eigenvalue is one-dimensional;
and every measurable eigensection carried by $L$ agrees a.e.\ with a real-analytic
nowhere-vanishing one.
\end{theorem}

\begin{proof}
(1) Since $\eu^{2\pi\ii\nu\cdot w}$ is $\Z^2$-periodic,
$(V_\nu F)(w):=\eu^{2\pi\ii\nu\cdot w}F(w)$ defines a linear bijection on measurable
sewn sections, preserves sections carried by $L$, and restricts to a unitary operator on
$\mathcal H_U$. Directly on measurable sewn sections,
\[
T_AV_\nu=\eu^{-2\pi\ii\nu\cdot\tau}V_\nu T_A,
\]
so $F_\nu=V_\nu F_0$ has eigenvalue $\lambda_\nu$.

(2) Define, before using these scalar coefficients,
\[
f_0:=\ip{e_{\mathrm{an}}}{F_0},\qquad
f_\nu(w):=\eu^{2\pi\ii\nu\cdot w}f_0(w).
\]
By Lemma~\ref{lem:scalar}, $f:=\ip{e_{\mathrm{an}}}{F}$ satisfies (C1)--(C3); $m=0$
makes $f$ and $f_0$ measurable functions on $\T^2$, with $f$ nonzero a.e.\ and $f_0$
nowhere zero. Set $r:=f/f_0$; then $r$ is measurable, nonzero a.e., and
$r(Sw)=\zeta\,r(w)$ a.e.\ with $\zeta:=\lambda_0/\lambda$.

If $\abs\zeta\neq1$, then $u:=\log\abs r$ is measurable and finite a.e.\ with
$u(S^nw)=u(w)+n\log\abs\zeta$ for every $n\ge1$ and a.e.\ $w$; choosing $M$ with
$\mu\{\abs u\le M\}>0$, Poincar\'e
recurrence returns a.e.\ point of that set to it infinitely often, contradicting
$\abs{u(S^nw)}\to\infty$ (this is the argument of Lemma~\ref{lem:carried}). So
$\abs\zeta=1$ and $\abs\lambda=\abs{\lambda_0}$.

Then $\abs r$ is measurable, $S$-invariant and descends to $\T^2$ (again $m=0$), hence
a.e.\ equal to a constant $c>0$ by ergodicity. Write $r=c\varphi$ with $\abs\varphi=1$
a.e.\ and $\varphi\circ S=\zeta\varphi$. Being bounded, $\varphi\in L^2(\T^2)$, and
comparing Fourier coefficients gives
$\hat\varphi_{\nu'}\bigl(\eu^{2\pi\ii\nu'\cdot\tau}-\zeta\bigr)=0$ for all $\nu'$. Some
coefficient is nonzero, so $\zeta=\eu^{2\pi\ii\nu\cdot\tau}$ for that $\nu$; rational
independence of $1,\alpha,\beta$ makes $\nu'\mapsto\eu^{2\pi\ii\nu'\cdot\tau}$ injective,
so exactly one mode survives and $\varphi=\hat\varphi_\nu\eu^{2\pi\ii\nu\cdot w}$. Hence
$\lambda=\lambda_0\eu^{-2\pi\ii\nu\cdot\tau}$ and $f=c\hat\varphi_\nu f_\nu$. Density of
$\{\eu^{-2\pi\ii\nu\cdot\tau}\}$ in $S^1$ follows because its closure is a closed subgroup,
finite of order $M$ only if $M\alpha,M\beta\in\Z$.
\end{proof}

\begin{theorem}[dichotomy]\label{thm:dich}\leavevmode\par
Let $(U_1,U_2,A)$ be strip-admissible with values in $U(2)$, $GL(2,\C)$, admitting a
dominated splitting $E^s\oplus E^u$, over $\tau\in\mathrm{DC}(c,\nu)$. Let $L$ be a
continuous invariant line with charge $m=c_1(L)[\T^2]$ and multiplier winding $k$.
By Lemma~\ref{lem:onlytwo}, $L\in\{E^s,E^u\}$. Then
\[
L\ \text{carries a nonzero measurable eigensection}
\quad\Longleftrightarrow\quad m=0\ \text{ and }\ k=0 ,
\]
and when both vanish the point spectrum realised on $L$ is exactly the dense coset
$\Lambda_L$ of Theorem~\ref{thm:class}.
\end{theorem}

\begin{proof}
($\Leftarrow$) Theorem~\ref{thm:Bex}, with Theorem~\ref{thm:class} for the description.
($\Rightarrow$) If $m\neq0$, Theorem~\ref{thm:B}. If $m=0$ and $k\neq0$, then
Lemma~\ref{lem:analytic} supplies \textup{(H2)} --- it assumes nothing about the charge ---
and Proposition~\ref{prop:winding} applies.
\end{proof}

\begin{corollary}[operator level; sharpening of Corollary~\ref{cor:C}]\label{cor:D}
Under the hypotheses of Theorem~\ref{thm:dich}, write $(m_s,k_s)$, $(m_u,k_u)$ for the
invariants of the two lines. Then
\[
\sigma_{\mathrm{pt}}(T_A)
=\bigsqcup_{i\in\{s,u\}\,:\ (m_i,k_i)=(0,0)}\Lambda_{E^i},
\]
so $T_A$ has an eigenvalue if and only if at least one line has both invariants zero. All
eigenvalues are geometrically simple in $\mathcal H_U$, i.e.
$\dim\ker(T_A-\lambda)=1$ for every eigenvalue $\lambda$. Within each qualifying line,
geometric simplicity is
Theorem~\ref{thm:class}; eigenvalues contributed by distinct qualifying lines cannot
coincide because their radii differ,
$\eu^{\lambda_s}<\eu^{\lambda_u}$ \textup{(Lemma~\ref{lem:domcons}(3))}. By the
Whitney budget $m_s+m_u=c_1(V)[\T^2]$ \textup{(Remark~\ref{rem:whitney})}, a bundle with
$c_1(V)\neq0$ admits \emph{at most} one qualifying line. Hence on such a bundle
\[
\sigma_{\mathrm{pt}}(T_A)\neq\emptyset
\iff \text{one of }E^s,E^u\text{ has }m=0\ \text{\emph{and}}\ k=0 ,
\]
and if a particular line is independently known to be topologically trivial, this reduces
to testing $k=0$ on that line. The budget bounds the number of charge-free lines from
above and never produces one: see Remark~\ref{rem:notrivial}.
\end{corollary}

\begin{proof}
Combine Lemma~\ref{lem:carried}, Theorem~\ref{thm:dich} and Theorem~\ref{thm:class}.
\end{proof}

\begin{remark}[the Chern budget does not produce a charge-free line]\label{rem:notrivial}
On a bundle with $c_1(V)[\T^2]=1$ the identity $m_s+m_u=1$ forbids \emph{two} charge-free lines
but does not supply one: $(m_s,m_u)=(-1,2)$ is equally admissible, and then no invariant
line is topologically trivial. The distinction is not academic. Take
\[
U_1\equiv I,\qquad
U_2(w)=\diag\bigl(\eu^{-4\pi\ii w_1},\ \eu^{2\pi\ii w_1}\bigr),\qquad
A(w)=\diag\bigl(t\,\eu^{-4\pi\ii\alpha w_2},\ \eu^{2\pi\ii\alpha w_2}\bigr),\quad t>1 .
\]
These are entire, strip-admissible and equivariant, with dominated splitting
$E^u=\C(1,0)$, $E^s=\C(0,1)$ and $D_1=1/t$; and $c_1(V)[\T^2]=1$. But \eqref{eq:K3} gives
$m(E^u)=2$ and $m(E^s)=-1$, so \emph{neither} line is topologically trivial, while
$p\equiv1$ and hence $k=0$ on both. Theorem~\ref{thm:dich} correctly returns an empty
point spectrum. The moral for Corollary~\ref{cor:D} and for the table of
Section~\ref{sec:zak}: that a charge-free line exists must be \emph{verified}, never
inferred from $c_1(V)$.
\end{remark}

\subsection*{Examples and limitations}

The following elementary examples delimit the theorem. All live on sewing systems with
$U_1\equiv I$ and diagonal $U_2$, for which \eqref{eq:K1} is immediate, and use
$\tau$ ergodic (for Example~\ref{ex:dominated}, Diophantine if Theorem~\ref{thm:B} is
to be quoted; its conclusion there can also be checked by hand).

\begin{example}[elliptic, unitary, rich point spectrum on a nontrivial bundle]
\label{ex:elliptic}
Take $n=2$, $U_2(w)=\diag(1,\eu^{-2\pi\ii w_1})$, so that $V=L_0\oplus L_1$ with
$c_1(L_0)[\T^2]=0$, $c_1(L_1)[\T^2]=1$, $c_1(V)[\T^2]=1$ (the topological type of the vector-Zak bundle
of Section~\ref{sec:zak}). The cocycle $C(w):=\diag\bigl(1,\ \eu^{-2\pi\ii\alpha
w_2}\bigr)$ is entire, bounded on every strip, unitary at real points, equivariant
(both diagonal entries satisfy the required quasi-periodicity: the first is periodic,
the second obeys $c(w+e_2)=\eu^{-2\pi\ii\alpha}c(w)$), and both Lyapunov exponents
vanish. The sections $F_{\nu}(w)=(\eu^{2\pi\ii\nu\cdot w},0)$, $\nu\in\Z^2$, are
analytic eigensections carried by $L_0$ with eigenvalues $\eu^{-2\pi\ii\nu\cdot\tau}$,
a dense subgroup of the circle. Hence: \emph{bundle-level} topology
($c_1(V)\neq0$), unitarity, and any arithmetic hypothesis on $\tau$ do NOT obstruct
point spectrum. The obstruction of Theorem~\ref{thm:A} is irreducibly line-resolved.
\end{example}

\begin{example}[dominated, analytic, point spectrum on the trivial line only]
\label{ex:dominated}
On the same bundle take $A_t(w):=\diag\bigl(t,\ \eu^{-2\pi\ii\alpha w_2}\bigr)$ with
$t>1$. This is strip-admissible, with dominated splitting $E^u=L_0$, $E^s=L_1$
($D_1=1/t$), and its two Lyapunov exponents are $\log t$ and $0$, hence their gap is
$\log t>0$. The eigensections
$F_\nu=(\eu^{2\pi\ii\nu\cdot w},0)$ persist, with eigenvalues
$t\,\eu^{-2\pi\ii\nu\cdot\tau}$, all carried by the trivial line $E^u$. On
$E^s=L_1$ ($c_1(E^s)[\T^2]=1$) the normalized data are computable by hand: $e=(0,1)$,
$q(w)=\eu^{-2\pi\ii\alpha w_2}$, hence $p\equiv1$, $k=0$, $g=0$, and (H2) is trivial;
Theorem~\ref{thm:A} excludes eigensections on $E^s$ (for Diophantine $\tau$ this is
also an instance of Theorem~\ref{thm:B}). The theorem is thus sharp as a
\emph{line-resolved} statement: eigensections do exist, but only on the trivial line.
\end{example}

\begin{example}[both lines nontrivial: empty point spectrum]\label{ex:both}
Take
\[
U_2(w)=\diag\bigl(\eu^{-2\pi\ii w_1},\ \eu^{2\pi\ii w_1}\bigr),\qquad
A(w)=\diag\bigl(2\,\eu^{-2\pi\ii\alpha w_2},\ \eu^{2\pi\ii\alpha w_2}\bigr),
\]
so that $V\cong L_1\oplus L_{-1}$ with $c_1(V)[\T^2]=0$; the data are
strip-admissible, dominated ($D_1=1/2$), with $c_1(E^u)[\T^2]=1$ and
$c_1(E^s)[\T^2]=-1$. In both
lines $p\equiv1$, $k=0$, $g=0$, so (H2) is trivial and Theorem~\ref{thm:A} applies to
each; by Lemma~\ref{lem:carried}, $T_A$ has \emph{no} measurable eigensections, for
every ergodic $\tau$ (no Diophantine condition needed here). This shows
Corollary~\ref{cor:C}(2) is not vacuous, and that the mechanism can force empty point
spectrum even on a topologically trivial ambient bundle.
\end{example}

\begin{remark}[vector-Zak equal-exponent construction]
An equal-exponent unitary example with the features of Example~\ref{ex:elliptic} can
also be built on the vector-Zak sewing itself (with its non-diagonal $U_2$), using the
explicit smooth unit section of \cite{FPVV2026}. Since Example~\ref{ex:elliptic} makes
the same point self-containedly, we do not use this auxiliary construction.
\end{remark}

\subsection*{Sharpness: the Diophantine hypothesis cannot be removed}

Theorem~\ref{thm:dich} is a statement about the strip-admissible dominated regime over
$\tau\in\mathrm{DC}(c,\nu)$. Within it, vanishing of the two invariants is necessary and
sufficient. It is not sufficient in general: the following shows that
$\mathrm{DC}$ genuinely enters the existence direction, and --- consistently with
Remark~\ref{rem:asym} --- that what fails is the \emph{modulus} equation, not
\textup{(H2)}.

\begin{theorem}[$\mathrm{DC}$ is not removable]\label{thm:sharp}
There is a continuous unitary sewing system on $\T^2$ with $c_1(V)[\T^2]=1$, a strip-admissible
analytic cocycle $A$ with values in $GL(2,\C)$ admitting a dominated splitting, and an
ergodic (Liouville) $\tau$, such that $L=E^u$ has $m(L)=0$ and $k(L)=0$ and satisfies
\textup{(H2)} trivially, yet $L$ carries no nonzero measurable eigensection.
\end{theorem}

\begin{proof}
Take the sewing of Examples~\ref{ex:elliptic}--\ref{ex:dominated}: $U_1\equiv I$,
$U_2(w)=\diag(1,\eu^{-2\pi\ii w_1})$, so $V=L_0\oplus L_1$ and
$c_1(V)[\T^2]=1$. Let
$r\colon\T^2\to\R$ be real analytic, depending on $w_1$ only, with $\langle r\rangle=0$,
and set
\[
A(w)=\diag\bigl(t\,\eu^{r(w)},\ \eu^{-2\pi\ii\alpha w_2}\bigr),\qquad
t>\eu^{\norm r_\infty}.
\]
Equivariance \eqref{eq:K2} holds (the first diagonal entry must be $\Z^2$-periodic, the
second must satisfy $a_2(w+e_2)=\eu^{-2\pi\ii\alpha}a_2(w)$); the data are
strip-admissible; and $D_1=\eu^{-r}/t<1$, so $E^u=L_0$ and $E^s=L_1$. On $L_0$ the section
$e=(1,0)$ satisfies \eqref{eq:K3} with $m=0$, and $q=t\eu^{r}>0$, so $p\equiv1$, $k=0$,
$g\equiv0$: \textup{(H2)} is trivially true.

The parameters $\alpha,\beta$ and the function $r$ are chosen below; fix those ultimate
choices. They make $1,\alpha,\beta$ rationally independent, so the resulting translation
$S$ is ergodic before the pointwise ergodic theorem is invoked.

Suppose $L_0$ carried a nonzero measurable eigensection with eigenvalue $\lambda$. By
Lemma~\ref{lem:scalar} there is a measurable $\Z^2$-periodic $f\neq0$ a.e.\ with
$qf=\lambda f\circ S$, so $u:=\log\abs f$ is measurable and finite a.e.; no assumption
that $u\in L^1$ is made or needed. After enlarging the exceptional null set by all of its
forward and backward $S$-iterates, work on one $S$-invariant conull set on which
\[
u\circ S^n-u=nc+\sum_{j<n}r\circ S^j\qquad(n\ge1),
\qquad c:=\log(t/\abs\lambda).
\]
For any measurable $u$ the family $\{u\circ S^n-u\}_{n\ge1}$ is tight in measure, since
$u\circ S^n$ and $u$ are equidistributed:
$\mu\{\abs{u\circ S^n-u}>2M\}\le2\mu\{\abs u>M\}$. Thus, writing
$X_n:=u\circ S^n-u$, tightness implies $X_n/n\to0$ in measure: for any $\eta>0$, choose
$M$ so that $\sup_n\mu\{\abs{X_n}>M\}<\eta$, and then take $n$ with $M/n<\eta$.
Since $r\in L^1$ has zero mean and $S$ is ergodic,
the pointwise ergodic theorem \cite{Walters1982} gives
$\tfrac1n\sum_{j<n}r\circ S^j\to0$ a.e., so the displayed iterated identity gives
$X_n/n\to c$ a.e.\ and hence in measure. Uniqueness of limits in measure forces $c=0$.
Hence the Birkhoff sums
$S_nr:=\sum_{j<n}r\circ S^j=u\circ S^n-u$ would be tight in measure.

It remains to choose $r$ and $\alpha$ for which they are not. Let $p_j/Q_j$ denote the
convergents of a continued fraction for $\alpha$, and put $\varepsilon_j:=\eu^{-Q_j}$.
Define
\[
r(w)=\sum_{j\ge1}\varepsilon_j\cos(2\pi Q_jw_1).
\]
Then $r$ extends holomorphically and boundedly to
$\abs{\operatorname{Im}w_1}<\rho$ for every $\rho<1/2\pi$, so it is real analytic and the
data are strip-admissible. Put
\[
\delta_j:=\norm{Q_j\alpha}_{\R/\Z},\qquad
B_j:=\frac{\varepsilon_j}{\delta_j}.
\]
The standard convergent estimates give
\[
\frac1{Q_{j+1}+Q_j}<\delta_j<\frac1{Q_{j+1}},
\qquad
\varepsilon_jQ_{j+1}<B_j<\varepsilon_j(Q_{j+1}+Q_j).
\]
Once $Q_{i+1}$ has been fixed, define the tail-independent envelope
\[
M_i:=\varepsilon_i(Q_{i+1}+Q_i).
\]
Thus $B_i<M_i$ for every later completion of the continued-fraction tail. At stage $j$,
$Q_j$ and the earlier envelopes are already fixed, while the next partial quotient can
make $Q_{j+1}$ arbitrarily large. Choose it so that
\begin{equation}\label{eq:liougrowth}
\text{(a)}\ \ Q_{j+1}\ge Q_j+5,
\qquad\qquad
\text{(b)}\ \ \varepsilon_jQ_{j+1}
>100\,j\Bigl(1+\sum_{i<j}M_i\Bigr) .
\end{equation}
After fixing $Q_{j+1}$, define $M_j$ as above. The lower bound for $B_j$ and the envelopes
for the earlier terms give
\[
B_j>100\,j\Bigl(1+\sum_{i<j}M_i\Bigr)
>100\,j\Bigl(1+\sum_{i<j}B_i\Bigr).
\]
Moreover every inequality imposed at an earlier stage persists under every later tail
choice. Condition~\textup{(b)} at stage $i$ is tail-independent once $Q_{i+1}$ is fixed;
the derived inequality for $B_i$ persists because $B_i$ stays above the fixed lower
bound $\varepsilon_iQ_{i+1}$, while each earlier $B_h$ on its right stays below its
already fixed envelope $M_h$. This is the required persistence mechanism.

In particular $B_j>100j$, so $B_j\to\infty$, and
$\delta_j=\varepsilon_j/B_j=o(Q_j^{-N})$ for every fixed $N$. Hence $\alpha$ is
Liouville, since for the convergent $p_j/Q_j$,
\[
\Bigl\lvert\alpha-\frac{p_j}{Q_j}\Bigr\rvert=\frac{\delta_j}{Q_j}
< Q_j^{-n}
\]
for every fixed $n$ and all large $j$. For $\nu'_j=(Q_j,0)$,
\[
\norm{\nu'_j\cdot\tau}_{\R/\Z}=\delta_j=o(Q_j^{-N})
\]
for every fixed $N$, so these frequencies violate every $\mathrm{DC}(c,\nu)$ condition.
Finally choose $\beta$ outside the rational span of $1,\alpha$; then
$1,\alpha,\beta$ are rationally independent and the translation $S$ is ergodic. Write
$D_n(x)=\sum_{i<n}\eu^{2\pi\ii ix}$, so that
\[
S_nr(w)=\sum_i\varepsilon_i\operatorname{Re}
\bigl[\eu^{2\pi\ii Q_iw_1}D_n(Q_i\alpha)\bigr],
\]
and recall $\abs{D_n(x)}=\abs{\sin\pi nx}/\abs{\sin\pi x}$ and
$\abs{D_n(x)}\le\min\{n,(2\norm x_{\R/\Z})^{-1}\}$. With
$n_j:=\lfloor1/(4\delta_j)\rfloor$ one has $n_j\delta_j\le\tfrac14$ always, and
$n_j\delta_j>\tfrac14-\delta_j\ge\tfrac18$ once $\delta_j\le\tfrac18$. Since
$\abs{\sin\pi x}\le\pi\abs x$ and $\sin$ increases on $[0,\pi/4]$,
\begin{equation}\label{eq:mainterm}
\varepsilon_j\abs{D_{n_j}(Q_j\alpha)}
=B_j\,\delta_j\,\frac{\abs{\sin\pi n_j\delta_j}}{\abs{\sin\pi\delta_j}}
\ \ge\ B_j\,\frac{\sin(\pi/8)}{\pi}\ >\ \frac{B_j}{10},
\end{equation}
because $\sin(\pi/8)/\pi=0.1218\ldots>\tfrac1{10}$. \emph{The weaker constant is the
correct one here: the interval $n_j\delta_j\in[\tfrac18,\tfrac14]$ does not yield
$B_j/5$.} For the other two blocks,
\[
\sum_{i<j}\varepsilon_i\abs{D_{n_j}(Q_i\alpha)}\le\tfrac12\sum_{i<j}B_i
\le\tfrac{B_j}{200j}\le\tfrac{B_j}{200},
\qquad
\sum_{i>j}\varepsilon_i n_j\le\tfrac{\varepsilon_{j+1}}{2\delta_j}\le\tfrac{B_j}{200},
\]
the first using $\abs{D_n(x)}\le(2\norm x_{\R/\Z})^{-1}$ together with the consequence
of \eqref{eq:liougrowth}(b) just proved, the second $\abs{D_n}\le n$ together with
$\varepsilon_{j+1}\le\eu^{-5}\varepsilon_j<\varepsilon_j/100$, which is
\eqref{eq:liougrowth}(a). So $S_{n_j}r$ is a single cosine
$a_j\cos(2\pi Q_jw_1+\theta_j)$ with $a_j>B_j/10$, plus an error of modulus at most
$B_j/100$. The set where $\abs{\cos(2\pi Q_jw_1+\theta_j)}\ge\tfrac12$ has measure
$\tfrac23$ (the integrand is $1/Q_j$-periodic in $w_1$ and $Q_j\in\N$), and there
\[
\abs{S_{n_j}r}\ \ge\ \frac{B_j}{20}-\frac{B_j}{100}\ =\ \frac{B_j}{25}\ >\ \frac{B_j}{50}.
\]
Hence $\mu\{\abs{S_{n_j}r}\ge B_j/50\}\ge\tfrac23$ for all large $j$, while
$B_j\to\infty$. The family $\{S_nr\}$ is therefore not tight, no such $u$ exists, and
$L_0$ carries no measurable eigensection.
\end{proof}

\begin{remark}\label{rem:sharpmoral}
The obstruction in Theorem~\ref{thm:sharp} is arithmetic, not topological, and it lives
entirely in the modulus: the \emph{mean} of $\log\abs q$ obstructs nothing --- it merely
determines $\abs\lambda=\eu^{\lambda_L}$, as the preceding tightness argument shows is
forced in this example ---
whereas its \emph{fluctuation} carries the extra hypothesis that
Remark~\ref{rem:asym} isolates. The example also complements
the concluding open problem from the opposite side: it exhibits data satisfying \textup{(H2)}
whose charge-free line still carries nothing.
\end{remark}

\section{Vector-Zak bundles and time-frequency motivation}\label{sec:zak}

This section records the setting that motivated the theorems. It contains no new results
about linear independence of time-frequency systems.

For $z=(x,\omega)\in\R^2$ let $\pi(z)f(t)=\eu^{2\pi\ii\omega(t-x/2)}f(t-x)$ be the Weyl
time-frequency shift on $L^2(\R)$. Heil, Ramanathan and Topiwala conjectured
\cite{HRT1996} that every finite set of time-frequency shifts of a nonzero $f\in
L^2(\R)$ is linearly independent; the conjecture is true for lattice subsets
\cite{Linnell1999}, for all four-point sets with real-valued generators
\cite{GuanOkoudjou2026}, for four-point sets with ultimately positive generators
\cite{MalikiosisPoursalidis2025}, and for a Lean-certified family of four-point sets
with symplectic covolume exceeding one \cite{Oussa2026}. In a recent counterexample
preprint, Faulhuber, Petersen, van~Velthoven and Voigtlaender construct twelve
time-frequency shifts of a nonzero Schwartz function with a linear dependence
\cite{FPVV2026}. Oussa gives a computer-assisted four-point construction over the same
vector-Zak translation $\tau^*$ used below \cite{Oussa4pt2026}; Dai, Deng, Shi, Wu and
Yang independently give a four-point construction over a different, biquadratic base
translation \cite{DDSWY2026}. The three cited versions are recent and unrefereed at this
writing.

The mechanism of \cite{FPVV2026} is a \emph{vector-valued Zak transform}: for
configurations in (a translate of) the lattice $\Z\times\tfrac12\Z$, the dependency
equation for a finite Weyl polynomial becomes an eigensection equation
\[
B(z)\,F(Tz)=\lambda\,F(z),\qquad Tz=z-\tau,
\]
for a $2\times2$ matrix symbol $B$ acting on \emph{vector-Zak functions}: measurable
$F\colon\R^2\to\C^2$ with the sewing automorphy
\[
F(z+e_1)=U_1(\omega)F(z),\qquad
F(z+e_2)=U_2F(z),
\]
with
\[
U_1(\omega)=\eu^{\pi\ii\omega}\diag(1,-1),\qquad
U_2=\begin{pmatrix}0&1\\1&0\end{pmatrix},
\]
cf.\ \cite[\S3, (3.16)--(3.17)]{FPVV2026}; the substitution $w:=Tz$,
$A(w):=B(Sw)$ converts this into the $S$-picture \eqref{eq:eigen2} used here. The pair
$(U_1,U_2)$ is a continuous (entire, strip-bounded) unitary sewing system in the sense
of Definition~\ref{def:sewing}. Its determinant multipliers are
$\det U_1=-\eu^{2\pi\ii\omega}$, $\det U_2=-1$, and the computation of
Lemma~\ref{lem:degree} (applied to the determinant bundle, with the mirror-normalized
connection form $-2\pi\ii w_1dw_2$) gives
\[
c_1(V)[\T^2]=1 .
\]
For the counterexample constructions in \cite{FPVV2026,Oussa4pt2026} the base
translation is $\tau^*=(2^{1/3}-1,\,2^{2/3}-1)$, which satisfies
$\mathrm{DC}(1/14,2)$ by Lemma~\ref{lem:cubic}; the construction of
\cite{DDSWY2026} instead runs over
$\bigl(\tfrac13+10^{-12}\sqrt2,\ \tfrac13+10^{-12}\sqrt3\bigr)$, which is
Diophantine of exponent $3$ by a norm-form bound in $\Q(\sqrt2,\sqrt3)$
\cite[Lemma~5.1]{DDSWY2026}, hence also in the class of
Definition~\ref{def:DC}. In all three papers the symbols are finite sums of
exponentials $c\,\eu^{\ii a\cdot z}$ with real frequency vectors $a$---entire and
bounded on every strip, hence strip-admissible together with the sewing whenever
$\det$ is zero-free on the real slice (``off-band'' symbols).

In this setting our results say: \emph{for an off-band symbol with a dominated
splitting over $\tau^*$, every measurable eigensection of the transfer operator---%
equivalently, every vector-Zak solution of the dependency equation---is carried,
almost everywhere, by a topologically trivial invariant line} (Theorem~\ref{thm:B} and
Corollary~\ref{cor:C}; by $m_s+m_u=c_1(V)[\T^2]=1$, at most one of the two lines can be
trivial). When neither line is trivial, the transfer operator has empty point
spectrum. We emphasize what is \emph{not} claimed: nothing here decides linear
independence of any time-frequency configuration---dependencies can and do occur
through trivializable carrying lines (this is precisely how the eigensections behind
the counterexamples \cite{FPVV2026,Oussa4pt2026,DDSWY2026} arise; in
\cite{DDSWY2026} the triviality of the carrying line is explicitly
certified), the correspondence between
dependencies and eigensections involves further identifications for which we refer
to \cite{FPVV2026}, and the elliptic and non-dominated regimes are untouched.
The contribution of the present paper to this picture is the topological and
cohomological confinement of possible carrying lines, and---by
Theorem~\ref{thm:dich}---the fact that this confinement is exact: on an off-band
dominated symbol over $\tau^*$, an invariant line carries an eigensection precisely when
its Chern class and its multiplier winding both vanish.

\paragraph{Scope of the comparison with recent HRT work.}
The recent counterexample preprints give three constructive results:
Faulhuber--Petersen--van Velthoven--Voigtlaender construct a twelve-point linear
dependence from a smooth vector-Zak eigensection and exact coefficient data
\cite{FPVV2026}; Oussa constructs a four-point counterexample with a validated
certificate \cite{Oussa4pt2026}; and Dai--Deng--Shi--Wu--Yang independently construct
a four-point counterexample using an interval-certified topologically trivial dominated
line and a smooth eigensection \cite{DDSWY2026}. These conclusions rest on the proofs
and computational certificates in the cited works; the present manuscript neither
reproduces nor independently certifies them. The cited HRT counterexample and
linear-independence results remain external, and no Zak/HRT correspondence theorem is
formalized here. These existence constructions are mathematically different from the
obstruction and classification results proved in this paper.

\begin{openproblem}
Does there exist a continuous (or smooth, or analytic-with-Liouville-$\tau$) example
satisfying all hypotheses of Theorem~\ref{thm:A} except \textup{(H2)}, whose charged
line \emph{does} carry a nonzero measurable eigensection? By
Theorem~\ref{thm:B} such an example requires failure of domination, of
strip-admissibility, or of $\mathrm{DC}$. In the one-frequency, charge-free analogue,
degree-nonzero $C^1$ circle cocycles never admit such solutions for any irrational
frequency \cite{GLL1991,ILR1993}, so genuinely two-frequency or low-regularity
mechanisms would be needed. Theorem~\ref{thm:sharp} answers the mirror image of this
question: it exhibits analytic dominated data over a Liouville $\tau$ satisfying
\textup{(H2)} whose charge-free, winding-free line carries no eigensection, the failure
being in the modulus equation rather than in \textup{(H2)}.
\end{openproblem}

\section{Disclosure and acknowledgments}\label{sec:disclosure}

The author thanks Vignon Oussa for helpful discussions concerning his work on HRT
counterexamples. The author also thanks Javad Mashreghi for organizing IWOTA 2026---the
International Workshop on Operator Theory and its Applications, held at Universit\'{e}
Laval in Qu\'{e}bec City---and for the opportunity to discuss these ideas during the
meeting.

\subsection*{Scope of the Lean formalization}\label{sec:formalization}

Selected results have been formalized and kernel-checked in Lean. The conventional
first-Chern-class-to-charge identification and the Stein/Cousin-II input remain external,
the cited HRT counterexamples and the Zak/HRT correspondence lie outside the scope of the
Lean development, and no claim of complete formal verification is made. The private Lean
source and exact dependency record are available from the author upon reasonable request.

\section*{Declaration of generative AI and AI-assisted technologies in the manuscript
preparation process}

During the preparation of this work, the author used AI tools to assist with language
refinement, consistency checking, literature organization, and the development and review
of portions of the Lean formalization. The author reviewed and edited all resulting material
and takes full responsibility for the content of the article.

\begingroup
\raggedright
\bibliographystyle{amsplain}
\bibliography{references}
\endgroup

\medskip
\noindent\textsc{Ahmadreza Azimifard}\\
\noindent Harmonic Research \& Technologies, LLC.\\
\noindent New York, NY, USA\\
\noindent\textit{Email:} afard@harmonicrt.com

\end{document}